\documentclass[12pt]{amsart}
\usepackage{amssymb}
\usepackage{amsbsy}
\usepackage{amscd}

\usepackage[mathscr]{eucal}

\usepackage[utf8]{inputenc}

\DeclareSymbolFont{rsfs}{U}{rsfs}{m}{n}
\DeclareSymbolFontAlphabet{\mathrsfs}{rsfs}

\usepackage{a4wide}
\usepackage[all]{xy}
\usepackage{tikz-cd}
\usepackage{tikz}
\usetikzlibrary{arrows,decorations.pathmorphing,backgrounds,positioning,fit,petri}
\usetikzlibrary{decorations.markings}
\usetikzlibrary{patterns,calc}

\usepackage{verbatim}
\usepackage{version}
\usepackage{color}
\definecolor{darkspringgreen}{rgb}{0.09, 0.45, 0.27}
\definecolor{deepjunglegreen}{rgb}{0.0, 0.29, 0.29}
\newenvironment{NB}{
\color{red}{\bf NB}. \footnotesize
}{}
\newenvironment{NB2}{
\color{blue}{\bf NB2}. \footnotesize
}{}
\excludeversion{NB}
\excludeversion{NB2}
\excludeversion{NB3}
\makeatletter

\usepackage{url}
\usepackage{ifmtarg}

\usepackage{xr-hyper}
\usepackage[
pdftex,
bookmarks=true,
colorlinks=true,
citecolor=deepjunglegreen,
filecolor=deepjunglegreen,
debug=true,
unicode,
pdfnewwindow=true]{hyperref}
\usepackage{cleveref}

\crefname{Theorem}{Theorem}{Theorems}
\crefformat{Theorem}{Theorem~#2#1#3}
\crefname{section}{\S}{\S\S}
\crefformat{section}{\S#2#1#3} 
\crefmultiformat{section}{\S\S#2#1#3}{, #2#1#3}{, #2#1#3}{, #2#1#3}
\crefname{Lemma}{Lemma}{Lemmas}
\crefformat{Lemma}{Lemma~#2#1#3}
\crefname{Proposition}{Proposition}{Propositions}
\crefformat{Proposition}{Proposition~#2#1#3}
\crefname{Corollary}{Corollary}{Corollaries}
\crefformat{Corollary}{Corollary~#2#1#3}
\crefname{Definition}{Definition}{Definitions}
\crefformat{Definition}{Definition~#2#1#3}
\crefname{Remark}{Remark}{Remarks}
\crefformat{Remark}{Remark~#2#1#3}
\crefname{Remarks}{Remark}{Remarks}
\crefformat{Remarks}{Remark~#2#1#3}
\crefname{Conjecture}{Conjecture}{Conjectures}
\crefformat{Conjecture}{Conjecture~#2#1#3}
\crefname{figure}{Figure}{Figure}
\crefformat{figure}{Figure~#2#1#3}

\crefname{appendix}{Appendix}{Appendices}
\crefformat{appendix}{\S#2#1#3}
\crefformat{enumi}{#2#1#3}

\crefname{equation}{}{}

\usepackage{enumitem}
\usepackage{extpfeil}

\renewcommand{\thesubsection}{\thesection(\@roman\c@subsection)}
\newcounter{number}
\makeatother
\newtheorem{Theorem}[equation]{Theorem}

\newtheorem{Lemma}[equation]{Lemma}
\newtheorem{Proposition}[equation]{Proposition}

\theoremstyle{definition}

\theoremstyle{remark}
\newtheorem{Remark}[equation]{Remark}

\numberwithin{equation}{section}

\newcommand{\defeq}{\overset{\operatorname{\scriptstyle def.}}{=}}
\newcommand{\CC}{{\mathbb C}}
\newcommand{\ZZ}{{\mathbb Z}}

\newcommand{\proj}{{\mathbb P}}
\newcommand{\CP}{\proj}

\newcommand{\SL}{\operatorname{\rm SL}}

\newcommand{\GL}{\operatorname{GL}}
\newcommand{\PGL}{\operatorname{PGL}}

\newcommand{\g}{{\mathfrak g}}

\newcommand{\Spec}{\operatorname{Spec}\nolimits}

\newcommand{\Hom}{\operatorname{Hom}}

\newcommand{\id}{\operatorname{id}}

\renewcommand{\MR}[1]{}
\newcommand{\Wedge}{{\textstyle \bigwedge}}

\newcommand{\bN}{\mathbf N}

\newcommand{\shfO}{\mathcal O}

\newcommand{\tslabar}{\mathbin{
\setbox0=\hbox{/\!\!/\!\!/}\rule[0.4\ht0]{\wd0}{.3\dp0}\kern-\wd0\box0}}

\newcommand{\Gr}{\mathrm{Gr}}

\newcommand{\cR}{\mathcal R}

\newcommand{\cT}{\mathcal T}
\newcommand{\cK}{\mathcal K}
\newcommand{\cO}{\mathcal O}
\newcommand{\cW}{\mathcal W}
\newcommand{\cM}{\mathcal M}
\newcommand{\uS}{\mathscr S}
\newcommand{\uQ}{\mathscr Q}

\makeatletter
\newcommand{\cA}[1][{}]{%
  \@ifmtarg{#1}%
  {\mathcal A}
  {\mathcal A(#1)}
}
\newcommand{\cAh}[1][{}]{%
  \@ifmtarg{#1}%
  {\mathcal A_\hbar}
  {\mathcal A_\hbar(#1)}
}
\makeatother

\newcommand{\cE}{\mathcal E}
\newcommand{\bv}{\mathbf v}
\newcommand{\bw}{\mathbf w}
\newcommand{\bz}{\mathbf z}
\newcommand{\sfu}{\mathsf u}
\newcommand{\sfy}{\mathsf y}

\newcommand{\BA}{\mathbb A}

\newcommand{\oZ}{\mathring{Z}}

\newcommand{\equi}{t}
\newcommand{\Equi}{T}
\newcommand{\formal}{t}

\newcommand{\arxiv}[1]{\href{http://arxiv.org/abs/#1}{\tt arXiv:\nolinkurl{#1}}}

\newcommand{\Weyl}{\mathbb W}

\newcommand{\altG}{G_\bullet}
\newcommand{\altN}{\bN_\bullet}
\newcommand{\cP}{\mathcal{P}}

\newcommand{\altGO}{G_{\bullet, \cO}}

\begin{document}

\title[Coulomb branches of quiver gauge theories with symmetrizers]
{Coulomb branches of quiver gauge theories with symmetrizers}

\author[H.~Nakajima]{Hiraku Nakajima}
\address{Kavli Institute for the Physics and Mathematics of the Universe (WPI),
  The University of Tokyo,
  5-1-5 Kashiwanoha, Kashiwa, Chiba, 277-8583,
  Japan
}
\email{hiraku.nakajima@ipmu.jp}

\author[A.~Weekes]{Alex Weekes}
\address{Perimeter Institute for Theoretical Physics,
	31 Caroline St. N., Waterloo, Ontario, N2L 2Y5, 
	Canada
}
\email{alex.weekes@gmail.com}

\subjclass[2000]{}
\begin{abstract}
  We generalize the mathematical definition of Coulomb branches of
  $3$-dimensional $\mathcal N=4$ SUSY quiver gauge theories in
  \cite{2015arXiv150303676N,2016arXiv160103586B,2016arXiv160403625B}
  to the cases with \emph{symmetrizers}. We obtain generalized affine Grassmannian slices of type $BCFG$ as examples of the
  construction, and their deformation quantizations via truncated shifted Yangians.   Finally, we study modules over these quantizations and relate them to
  the lower triangular part of the quantized enveloping algebra of type $ADE$.
  \begin{NB}
    The final sentence is added on May 15.
  \end{NB}%
  \begin{NB2}
  July 4: Changed the final two sentences.  Previous version:
  
  We obtain generalized slice in affine Grassmannian of type BDFG as examples of the constrcution. We also study their quantizations, and relate them to the lower triangular part of the quantized enveloping algebra of type ADE. 
  \end{NB2}
\end{abstract}

\maketitle




\section{Introduction}

Let $I$ be a finite set.
Recall $(c_{ij})_{i,j\in I}$ is a \emph{symmetrizable Cartan matrix}
if \begin{itemize}
\item $c_{ii} = 2$ for all $i\in I$, and $c_{ij}\in\ZZ_{\le 0}$ for all $i\neq j$,
\item there is $(d_i)\in \ZZ_{>0}^I$ such that $d_i c_{ij} = d_j c_{ji}$ for all $i,j$.
\end{itemize}
When $d_i = 1$ for any $i\in I$, a mathematical definition of the
Coulomb branch of a $3d$ $\mathcal N=4$
\begin{NB}
  Added on Dec. 28.
\end{NB}%
quiver gauge theory associated
with two $I$-graded vector spaces $V = \bigoplus V_i$,
$W = \bigoplus W_i$ was given in
\cite{2015arXiv150303676N,2016arXiv160103586B}, and its properties
were studied in \cite{2016arXiv160403625B}.
In this note, we generalize the definition to more general
symmetrizable cases. This new definition is motivated by works of
Geiss, Leclerc and Schr\"oer (\cite{MR3660306} and the subsequent papers \cite{MR3801499,MR3555157,MR3848021,MR3830892,arXiv181209663})
which aim to generalize various results on relations between symmetric
Kac-Moody Lie algebras and quivers to symmetrizable cases. They modify
quiver representations by replacing vector spaces on vertices by free
modules of truncated polynomial rings. They use different variables
for polynomials, which are related to each other according to $d_i$. This
modification allows them to relate quiver representations to
symmetrizable Kac-Moody algebras.  Their work, and ours, is also partly motivated by the theory of modulated graphs \cite{DR80,NT2016}, another approach to quivers in symmetrizable types.
\begin{NB2}
July 4: Added final sentence above, as Nandakumar-Tingley was my original motivation; this could be mentioned elsewhere if that seems preferable.
\end{NB2}

In \cite{2016arXiv160103586B} we assign vector bundles over the formal
disk $D = \operatorname{Spec} \CC[[z]]$. Since we can take different
variables $z_i$ for each vertex $i\in I$, the definition has the same
modification.

\begin{NB}
  Add on Nov.~30:
\end{NB}%

A similar construction was considered in the context of $4d$
$\mathcal N=2$ quiver gauge theories by Kimura and Pestun
\cite{Kimura:2017hez} under the name of \emph{fractional quiver gauge
  theories}.

Let us recall that we defined Coulomb branches of quiver gauge
theories associated with a \emph{symmetrizable} Cartan matrix in a
different way in \cite[\S4]{2016arXiv160403625B}. There, we realize a
symmetrizable Cartan matrix by a \emph{folding} of a graph. This folding gives a finite group action on the Coulomb branch of the quiver gauge theory of the (unfolded) graph.  Then we may define the Coulomb branch
of the symmetrizable theory as the corresponding fixed point subscheme. This
construction recovers the twisted monopole formula by Cremonesi,
Ferlito, Hanany and Mekareeya \cite{Cremonesi:2014xha}, as the Hilbert
series of the coordinate ring. This gives supporting evidence that
the folding construction is a reasonable candidate for a mathematical definition
of the Coulomb branch.

Our new construction also gives the twisted monopole formula.
It is natural to believe that the folding construction and the new one
give isomorphic varieties. However various properties of the Coulomb
branch are obvious in the new construction, while they are not in the
old one.
\begin{NB}
    Added on Mar.~6.
\end{NB}%
For example, the twisted monopole formula requires a proof in the
old construction, while it is obvious in the new
construction. \begin{NB2} I suppose e.g.~normality and being
    symplectic on the smooth locus are also quite non-trivial in the
    old construction? \end{NB2}%
We also do not know how to show the
normality in the old construction, while the proof in
\cite{2016arXiv160103586B} works for the new construction.
\begin{NB}
  The regular locus of the fixed point subscheme contains the fixed
  point subscheme in the regular locus. The symplectic form remains a
  symplectic form on the latter locus. I cannot rule out a possibility
  to have a regular point in the fixed point subscheme contained in
  the singular locus of the original Coulomb branch.
\end{NB}%
Therefore we believe that the new
construction has its own meaning.
In addition, work in progress of de Campos Affonso will identify the new definition with the symmetric bow
varieties introduced in \cite{Henrique} for quiver gauge theories of
non-symmetric affine Lie algebras of classical type. This identification is not clear for the old construction of the Coulomb branch as a fixed point subscheme. 

\begin{NB2}
July 5: Added the following paragraph
\end{NB2}

In fact, we will give a second potential definition for the Coulomb branch of a quiver gauge theory with symmetrizers in \cref{section: second definition}.  In many cases both definitions agree, and in particular this is true in finite BCFG types.  However, in general type they are different.  This alternative definition applies to more general data than quivers with symmetrizers, which may be of independent interest.

\begin{NB}
  The following two paragraphs are added on May 15.
\end{NB}%

As a generalization of one of the main results in
\cite{2016arXiv160403625B}, we show that our Coulomb branches are
generalized slices in the affine Grassmannian when the Cartan matrix
is of type $BCFG$ (\cref{thm:slice}). Therefore the geometric Satake
correspondence, as modified in \cite{2017arXiv170900391K}, says that
the direct sum of hyperbolic stalks of the intersection cohomology
complexes of our Coulomb branches has a structure of a finite
dimensional irreducible representation of the Langlands dual Lie
algebra. We expect that the same should be true for arbitrary
symmetrizable Kac-Moody Lie algebras as a symmetrizable generalization
of the conjecture in \cite[\S3(viii)]{2016arXiv160403625B}. (See also
\cite{2018arXiv181004293N} for a refinement of the conjecture.)

Also as a generalization of the main result in seven authors'
(BFK${}^2$NW${}^2$) appendices of \cite{2016arXiv160403625B} and also
of \cite{Weekes}, we show that the quantization of the Coulomb branch
is a truncated shifted Yangian when the Cartan matrix is of type
$BCFG$. (See \cref{thm:quant}.) Its modules can be analyzed by using
techniques of the localization theorem in equivariant homology groups,
even though we use infinite dimensional varieties
\cite{MR3013034,2019arXiv190405415W,modules}\begin{NB2} July 4: Added citation of Webster's paper \end{NB2}. We study the fixed point subvariety with
respect to a $\CC^\times$-action in an infinite dimensional variety
used in the definition of the Coulomb branch in \cref{sec:fixed}. It
turns out that the fixed point subvariety is the same as one appears
in the Coulomb branches of type $ADE$, which is a disjoint union of
varieties appearing Lusztig's work on canonical bases of
$\mathbf U_q^-$ of type $ADE$ \cite{Lu-can2}.
This implies that a certain category of modules of the truncated shifted
Yangian of type $BCFG$ \emph{categorifies} $\mathbf U_q^-$ of type
$ADE$. (See \cref{thm:ADE}. We only explain a parametri\-zation of
simple modules for simplicity.) It is interesting to understand the
relation between this analysis and the geometric Satake correspondence
explained above, as we obtain different Lie algebras, type $ADE$ and
$BCFG$.

\begin{NB}
    The following two paragraphs are added on Mar.~6.
\end{NB}%

Let us also remark that our construction can be applied to more
general situations than considered here. For example, the first-named
author originally introduced the Coulomb branch via cohomology with
compact support of the moduli space of twisted maps from $\proj^1$ to
the Higgs branch $\mathcal M_H$ (viewed as a quotient stack) with
coefficients in the sheaf of vanishing cycles
\cite{2015arXiv150303676N}. This definition can be generalized to our
setting, just changing the domain $\proj^1$ for each vertex $i\in I$.
This view point might shed a new light on the Higgs branch
$\mathcal M_H$ corresponding to our new construction: we cannot make sense of $\mathcal M_H$, but the space
of maps to $\mathcal M_H$ \emph{does} make sense. In particular,
enumerative problems for $\mathcal M_H$, such as discussed in
\cite{MR3752463}, are meaningful.

We also hope that our view point is useful to make advance in the
program of Geiss, Leclerc, Schr\"oer. We may hope to use the above
space of maps to $\mathcal M_H$ to realize representations of the Lie
algebra, or its cousins the Yangian and the quantum loop algebra associated with
the symmetrizable Cartan matrix $(c_{ij})$.

The paper is organized as follows. In \cref{sec:definition} we give
the definition of Coulomb branches for symmetrizable Cartan matrix
$(c_{ij})$.  Since it is a modification of the original one in
\cite{2016arXiv160103586B}, we only explain where we change the definition.
In \cref{sec:example} we determine Coulomb branches in some cases when the Cartan
matrix is $2\times 2$, and there are no framed vector spaces
$W_i$.
In \cref{sec:slice} we show that Coulomb branches are generalized
slices in the affine Grassmannian when the Cartan matrix is of type
$BCFG$. The proof is the same as in \cite[\S3]{2016arXiv160403625B},
once examples in \cref{sec:example} are determined.
In \cref{sec:quant} we discuss quantized Coulomb branches. We show
that they are isomorphic to truncated shifted Yangians in type $BCFG$.
In \cref{sec:app} we give an explicit presentation of the coordinate
ring of the zastava space of degree $\alpha_1+\alpha_2$ of type $G_2$.
This is used in \cref{sec:example}.
Contents in \cref{sec:fixed} are already explained above.  
In \cref{section: second definition} we present a second possible definition for the Coulomb branch associated to a quiver with symmetrizers, as mentioned above.  

\begin{NB}
  Added on May 15.
\end{NB}

\begin{NB2}
July 4: Reference to \cref{section: second definition} added.
\end{NB2}
\begin{NB2}
July 5: Changed final sentence, and added explanatory paragraph above instead
\end{NB2}

\begin{NB2}
July 12: Changed reference to \cref{sec:example} to say ``some cases'''
\end{NB2}

\section{Definition}\label{sec:definition}

\subsection{A valued quiver}
\label{subsection: A valued quiver}

Let $(c_{ij})_{i,j\in I}$ be a symmetrizable Cartan matrix.  We assign a \emph{valued} graph where it has vertices $i\in I$ and
unoriented edges between $i$, $j$ for $c_{ij} < 0$ with values
$(|c_{ij}|, |c_{ji}|)$. A \emph{valued} quiver is a valued graph
together with a choice of an orientation of each edge. Following
\cite{MR3660306}, we set
\(
   g_{ij} = \gcd(|c_{ij}|, |c_{ji}|),
\)
\(
   f_{ij} = |c_{ij}|/g_{ij}
\)
when $c_{ij} < 0$. Note that these are independent of $d_i$.

We take the formal disk $D_i = \operatorname{Spec} \CC[[z_i]]$ for
each vertex $i\in I$. For a pair $(i,j)$ with $c_{ij} < 0$ we take the
formal disk $D = \operatorname{Spec} \CC[[z]]$ and consider its branched coverings $\pi_{ji}\colon D_i\to D$,
$\pi_{ij} \colon D_j \to D$ defined by the maps $\pi_{ji}^\ast(z) = z_i^{f_{ij}}$,
$\pi_{ij}^\ast(z) = z_j^{f_{ji}}$ of coordinate rings.
\begin{NB}
  Corrected from $z_i^{f_{ji}}$, $z_j^{f_{ij}}$.
\end{NB}%
The disk $D$ depends on $(i,j)$, but
we drop $i$, $j$ from the notation. Let $D_i^*$, $D_j^*$, $D^*$ denote
the punctured formal disk for $D_i$, $D_j$, $D$ respectively.

\begin{Remark}
    In \cite{MR3660306} the relation (H2)
    $\varepsilon_i^{f_{ji}} \alpha_{ij}^{(g)} = \alpha_{ij}^{(g)}
    \varepsilon_j^{f_{ij}}$
    is imposed, where $\varepsilon_i$, $\varepsilon_j$ are edge loops
    at $i$ and $j$ respectively, and $\alpha_{ij}^{(g)}$ is the $g$-th
    arrow from $j$ to $i$. It means that we have
    $z_i^{f_{ji}} = z = z_j^{f_{ij}}$. Thus it differs from our
    convention by $f_{ij}\leftrightarrow f_{ji}$. This is probably
    compatible with geometric Satake correspondence: We will obtain
    generalized
    \begin{NB}
        Added on Mar.~6.
    \end{NB}%
    slices in the affine Grassmannian for $G$ for $(c_{ij})$ below,
    and hence representations of $G^\vee$, by the work of
    Krylov~\cite{2017arXiv170900391K}.
    \begin{NB}
        Added on Mar.~6.
    \end{NB}%
    On the other hand, the space of constructible functions on modules
    over the quiver with the relation (H2) is the enveloping algebra
    of the upper triangular subalgebra $\mathfrak n$ of the Lie
    algebra $\mathfrak g$ for $(c_{ij})$. Since we hope to compare
    representations of the same Lie algebra in Coulomb branches and
    \cite{MR3660306}, we need to take Langlands dual relation of (H2).

    Note also that the relation imposed in \cite{MR3500832} for a
    cluster algebra related to the quantum loop algebra
    $\mathbf U_q(\mathbf L{\mathfrak g})$ is the same as ours. See
    \cite[\S1.7.1]{MR3660306}. We believe that this is compatible with
    our results in \cref{sec:quant}, as the $K$-theoretic version of
    our construction in \cref{sec:quant} should yield
    $\mathbf U_q(\mathbf L{\mathfrak g})$-modules. However, the $K$-theoretic version of our construction
    does not immediately give a new approach to the results of
    \cite{MR3500832}. Modules obtained in this way are infinite
    dimensional, while \cite{MR3500832} discussed Kirillov-Reshetikhin
    modules, which are finite dimensional. Nevertheless we expect that
    it gives a first step towards in that direction.
\end{Remark}

\subsection{A moduli space}
\label{subsection: A moduli space}

Fix a valued quiver for $(c_{ij})_{i,j\in I}$.  Let $V=\bigoplus V_i$, $W=\bigoplus W_i$ be finite dimensional
$I$-graded complex vector spaces. Let $\bv_i = \dim V_i$,
$\bw_i = \dim W_i$.
We consider the moduli space $\cR$ parametrizing the following objects:
\begin{itemize}
\item a rank $\bv_i$ vector bundle $\cE_i$ over $D_i$ together with a
  trivialization $\varphi_i\colon \cE_i|_{D_i^*} \to
  V_i\otimes_\CC \shfO_{D_i^*}$ for $i\in I$,
\item a homomorphism $s_i\colon W_i\otimes_\CC\shfO_{D_i}\to \cE_i$
  such that $\varphi_i\circ (s_i|_{D_i^*})$ extends to $D_i$ for
  $i\in I$,
\item a homomorphism
  $s_{ij}\in \CC^{g_{ij}}\otimes_\CC
  \Hom_{\shfO_D}(\pi_{ij*}\cE_j,\pi_{ji*}\cE_i)$ such that
  $(\pi_{ij*}\varphi_i) \circ (s_{ij}|_{D^*}) \circ
  (\pi_{ji*}\varphi_j)^{-1}$ extends to $D$, where $c_{ij} < 0$ and
  there is an arrow $j\to i$ in the quiver.
\end{itemize}

The moduli space of pairs $(\cE_i,\varphi_i)$ as above is the affine
Grassmannian $\Gr_{\GL(V_i)}$ for $\GL(V_i)$.

Dropping the extension conditions in the second and third, we have a
larger moduli space $\cT$, which is an infinite rank vector bundle
over $\prod_i \Gr_{\GL(V_i)}$. Then $\cR$ is a closed subvariety in $\cT$.

When $c_{ij} = c_{ji}$, $\cR$ is nothing but the variety of triples
introduced in \cite[\S2(i)]{2016arXiv160103586B}.

Let $G = \prod_i \GL(V_i)$, $G_\cO = \prod_i \GL(V_i)[[z_i]]$,
$\Gr_G = \prod_i \Gr_{\GL(V_i)}$. 
\begin{NB}
  The last notation added on Dec.~28. Strictly speaking, we should
  emphasize that $\Gr_{\GL(V_i)}$ is defined for $z_i$. But it is too
  cumbersome if we put the dependence into the notation or explanation.
\end{NB}%
We have a $G_\cO$-action on $\cR$ by change of trivializations
$\varphi_i$, and we consider the $G_\cO$-equivariant Borel-Moore homology
group $H^{G_\cO}_*(\cR)$ with complex coefficients. This is defined
rigorously as a double limit as in
\cite[\S2(ii)]{2016arXiv160103586B}.

The spaces $\bN_\cO$, $\bN_\cK$ appear during the construction of the
convolution product in \cite[\S3(i)]{2016arXiv160103586B}. They were
the space of sections (resp.\ rational sections) of the vector bundle
associated with the \emph{trivial} $G$-bundle. In our setting,
$\bN_\cO$ is defined as the direct sum of
$\Hom_{\shfO_{D_i}}(W_i\otimes_\CC\shfO_{D_i},
V_i\otimes_\CC\shfO_{D_i})$ and
$\CC^{g_{ij}}\otimes_\CC
\Hom_{\shfO_D}(\pi_{ij*}(V_j\otimes_\CC\shfO_{D_j}),\pi_{ji*}(V_i\otimes_\CC\shfO_{D_i}))$. For
$\bN_\cK$, we take homomorphisms over $\shfO_{D_i^*}$ and
$\shfO_{D^*}$. We have maps $\Pi\colon \cR\to \bN_{\cO}$ and
$\cT\to\bN_{\cK}$. 
\begin{NB}
  The above paragraph is added on May 10.
\end{NB}%

\subsection{Twisted monopole formula}
\label{subsection: twisted monopole formula}

\begin{NB}
  The following two paragraphs are added on Dec.~28.
\end{NB}%
Recall that the monopole formula for the Hilbert series of the Coulomb
branch of a gauge theory \cite{Cremonesi:2013lqa} is interpreted as
the Poincar\'e polynomial of $H^{G_\cO}_*(\cR)$ with a suitable
modification in the ordinary untwisted case
\cite{2015arXiv150303676N,2016arXiv160103586B}.
The twisted monopole formula is given in \cite{Cremonesi:2014xha} to
cover Coulomb branches of quiver gauge theories for certain symmetrizable
Cartan matrices. It is of the same form
\(
   \sum_\lambda t^{2\Delta(\lambda)} P_G(t;\lambda)
\)
as the untwisted monopole formula, where the summation runs over the
set of dominant coweights $\lambda$ of the gauge group $G$, and
$P_G(t;\lambda)$ is the Poincar\'e polynomial of the equivariant
cohomology ring $H_{\mathrm{Stab}_G(\lambda)}^*(\mathrm{pt})$.
Only $\Delta(\lambda)$ is changed from the untwisted monopole
formula: if $i$, $j\in I$, the ordinary $2\Delta(\lambda)$ contains
contribution
\(
  |\lambda_i^a - \lambda_j^b|,
\)
where $(\lambda_i^a)_{a=1,\dots,\bv_i}$,
$(\lambda_j^b)_{b=1,\dots,\bv_j}$ are components of $\lambda$ for
vertices $i$, $j$ respectively.
In the twisted monopole formula, this contribution is simply replaced by
\(
  |f_{ji} \lambda_i^a - f_{ij} \lambda_j^b|.
\)

Let us check that our new $\cR$ gives the twisted monopole formula as
the Poincar\'e polynomial.  The argument is a simple modification of \cite[\S2(iii)]{2016arXiv160103586B}.  We do so under an additional assumption:
\begin{equation}
\label{assumption on Cartan}
\text{For all } i,j\in I, \text{ if } c_{ij} <0 \text{ then } f_{ij}=1 \text{ or } f_{ji} = 1
\end{equation}
In particular all finite types satisfy this assumption.
\begin{NB2} July 4: Added assumption as an equation, so that it can be referred to in Appendix C \end{NB2}

Let $\Gr_G^\lambda$ denote the $G_\cO$-orbit in $\Gr_G$ corresponding
to a dominant coweight $\lambda$ of $G$. Let $\cR_\lambda$,
$\cT_\lambda$ denote the inverse image of $\Gr_G^\lambda$ under the
projection $\pi\colon\cR\to\Gr_G$, $\pi\colon\cT\to\Gr_G$
respectively.
As in \cite[Lemma~2.2]{2016arXiv160103586B}, $\cT_\lambda/\cR_\lambda$
is a vector bundle over $\Gr_G^\lambda$. The fiber of $\cT_\lambda$ at $\lambda$ is
\begin{equation*}
  \bigoplus_{j\to i} \CC^{g_{ij}}\otimes z_i^{\lambda_i} z_j^{-\lambda_j}
  \Hom_{\CC[[z]]}(V_j\otimes\CC[[z_j]],V_i\otimes\CC[[z_i]]),
\end{equation*}
while the fiber of $\cR_\lambda$ is its intersection with
$\bigoplus_{j\to i} \CC^{g_{ij}}\otimes
  \Hom_{\CC[[z]]}(V_j\otimes\CC[[z_j]],V_i\otimes\CC[[z_i]])$.
Here $z = z_i^{f_{ij}} = z_j^{f_{ji}}$.
\begin{NB}
  Corrected from $z = z_i^{f_{ji}} = z_j^{f_{ij}}$.
\end{NB}%
Therefore the rank of $\cT_\lambda/\cR_\lambda$ is  
\begin{equation*}
  d_\lambda:=\sum_{i\to j} g_{ij}
  \sum_{a=1}^{\bv_i}\sum_{b=1}^{\bv_j}
  \max( f_{ij} \lambda_j^b - f_{ji} \lambda_i^a, 0).
\end{equation*}
\begin{NB2}
June 26: Is this formula correct when $f_{ij}, f_{ji} > 1$?  
\end{NB2}
Following the notation of \cite[Section 2(iii)]{2016arXiv160103586B}, we may formally write the Poincar\'e polynomial of $\cR$.  Let $\cR_{\leq \mu}$ denote the inverse image of the closure $\overline{\Gr_G^\mu} = \bigsqcup_{\lambda\leq \mu} \Gr_G^\lambda$ in $\cR$.  As in \cite[Proposition 2.7]{2016arXiv160103586B}:
\begin{Proposition}
\label{prop: eq: poincare polynomial}
The Poincar\'e polynomial for $\cR_{\leq \mu}$ is given by
$$
P_t^{G_\cO} (\cR_{\leq \mu}) = \sum_{\lambda \leq \mu} t^{2 d_\lambda - 4 \langle \rho, \lambda\rangle} P_G(t; \lambda)
$$
where the sum is over dominant coweights $\lambda$ with $\lambda \leq \mu$.
\end{Proposition}
In particular, taking the limit over $\mu$ we formally obtain:
\begin{equation}
\label{eq: poincare polynomial}
P_t^{G_\cO} (\cR) = \sum_\lambda t^{2 d_\lambda - 4 \langle \rho, \lambda\rangle} P_G(t; \lambda)
\end{equation}
However, we note that this expression may not converge even as a Laurent series.
\begin{NB}
  Let us consider the case $c_{12} = -1$, $c_{21} = -m$
  ($m\in \ZZ_{>0}$) as below. Then $g_{12} = 1$, $f_{12} = 1$,
  $f_{21} = m$. Hence we have
  \(
  \sum_{a=1}^{\bv_1}\sum_{b=1}^{\bv_2}
  \max( \lambda_2^b - m \lambda_1^a, 0).
  \)
  Thus our convention recovers the twisted monopole formula.
\end{NB}%

\begin{NB}
  Consider the case $f_{ji}=1$, hence $z_j=z$. Then
  \begin{equation*}
    \Hom_{\CC[[z]]}(\CC[[z_j]],\CC[[z_i]])
    \cong \Hom_{\CC[z_i]}(\pi_{ji}^*\CC[[z_j]], \CC[[z_i]])
    \cong \CC[[z_i]].
  \end{equation*}
  Then we consider the quotient
  \[ z_i^{\lambda_i} z_j^{-\lambda_j}\CC[[z_i]] /
    \left(z_i^{\lambda_i} z_j^{-\lambda_j}\CC[[z_i]]\cap
      \CC[[z_i]]\right) = z_i^{\lambda_i - f_{ij}\lambda_j}
    \CC[[z_i]]/ \left(z_i^{\lambda_i - f_{ij}\lambda_j} \CC[[z_i]]\cap
      \CC[[z_i]]\right).
  \]
  This is $\max(f_{ij}\lambda_j - \lambda_i,0)$-dimensional.

  Next consider the case $f_{ij}=1$, $z_i = z$. We have
  \begin{equation*}
      z_i^{\lambda_i} z_j^{-\lambda_j}
      \Hom_{\CC[[z]]}(\CC[[z_j]], \CC[[z_i]])
      = z_j^{f_{ji}\lambda_i - \lambda_j}
      \Hom_{\CC[[z]]}(\CC[[z_j]], \CC[[z]]).
  \end{equation*}
  Hence the answer depends only on $f_{ji}\lambda_i-\lambda_j$, and
  hence we may assume $\lambda_i=0$. Moreover we may assume
  $\lambda_j > 0$, as
  $\operatorname{rank} \cT_\lambda/\cR_\lambda = 0$ otherwise. We
  consider the codimension of the embedding
  \begin{equation*}
      \Hom_{\CC[[z]]}(\CC[[z_j]],\CC[[z]])
      \hookrightarrow \Hom_{\CC[[z]]}(z_j^{\lambda_j}\CC[[z_j]], \CC[[z]]).
  \end{equation*}
  Suppose $\lambda_j = f_{ji}k'+k''$ with $0\le k'' < m$. Then an
  element $\xi$ in the right hand side is contained in the left hand
  side if and only if $\xi(z_j^{\lambda_j})$ is divisible by $z^{k'}$,
  $\xi(z_j^{\lambda_j+1})$ is divisible by $z^{k'}$, \dots,
  $\xi(z_j^{\lambda_j+(f_{ji}-k'')}$ is divisible by $z^{k'+1}$, \dots
  $\xi(z_j^{\lambda_j+f_{ji}-1}$ is divisible by $z^{k'+1}$. In total
  we have $k'(f_{ji}-k'') + (k'+1)k'' = f_{ji}k'+ k'' = \lambda_j$
  equations.
  In other words, the quotient is
  \begin{multline*}
      \Hom_\CC(z_j^{\lambda_j}\CC[[z_j]]/z_j^{\lambda_j+f_{ji}-k''}\CC[[z_j]],
      \CC[[z]]/z^{k'}\CC[[z]])
\\
      \oplus
      \Hom_\CC(z_j^{\lambda_j+f_{ji}-k'}\CC[[z_j]]/z_j^{\lambda_j+f_{ji}-1}\CC[[z_j]],
      \CC[[z]]/z^{k'+1}\CC[[z]]).
  \end{multline*}
  Therefore the codimension is $\lambda_j$.
\end{NB}%

\begin{NB}
  The following is added on Dec.~28.
\end{NB}%
The monopole formula is closely related to this Poincar\'e polynomial: the contribution $\Delta(\lambda)$ mentioned above is given by
\begin{align*}
  \Delta(\lambda) &:= d_\lambda - 2 \langle \rho, \lambda\rangle - \frac12 \sum_{i\to j} g_{ij}
  \sum_{a=1}^{\bv_i}\sum_{b=1}^{\bv_j}
  (f_{ij} \lambda_j^b - f_{ji} \lambda_i^a) \\
  &= - \sum_i \sum_{1 \leq a < b \leq \bv_i} | \lambda_i^a - \lambda_i^b | + \frac{1}{2} \sum_{i\to j} g_{ij}
  \sum_{a=1}^{\bv_i}\sum_{b=1}^{\bv_j}
  |f_{ij} \lambda_j^b - f_{ji} \lambda_i^a|
\end{align*}
The difference \(\frac12 \sum_{i\to j} g_{ij}
  \sum_{a=1}^{\bv_i}\sum_{b=1}^{\bv_j}
  (f_{ij} \lambda_j^b - f_{ji} \lambda_i^a)\) depends only on the sums $\sum_a \lambda_i^a$, $\sum_b \lambda_j^b$.   In particular, it is possible to view the twisted monopole formula as the Poincar\'e polynomial of $\cR$, but with respect to a {\em different} grading (i.e.~different from the homological grading).
See \cite[Remark~2.8]{2016arXiv160103586B}.

\begin{NB2}
July 4: Added remark below -- is it too imprecise?
\end{NB2}
\begin{Remark}
If the assumption (\ref{assumption on Cartan}) does not hold, then the ranks of $\cT_\lambda / \cR_\lambda$ are generally not given by such a simple formula.  The corresponding Poincar\'e polynomial (and monopole formula) is thus more complicated.
\end{Remark}

\subsection{Convolution product}
\label{subsection: convolution product}

The definition of the convolution product on $H^{G_\cO}_*(\cR)$ goes
exactly as in \cite[3(iii)]{2016arXiv160103586B}. Moreover, we have an
algebra embedding $\bz^*\colon H^{G_\cO}_*(\cR)\to H^{G_\cO}_*(\Gr_G)$
as in \cite[\S5(iv)]{2016arXiv160103586B}, where
$\bz\colon\Gr_G\to \cR$ is the embedding.
The algebra $H^{G_\cO}_*(\Gr_G)$ is isomorphic to that which appears in the ordinary construction (i.e.~all $z_i$'s are replaced by $z$).
\begin{NB}
  We must be careful when we consider the quantization, as it is given
  by the loop rotation.
\end{NB}%
In particular, $H^{G_\cO}_*(\Gr_G)$ is commutative, and therefore
$H^{G_\cO}_*(\cR)$ is commutative as well. We define the Coulomb branch as
\begin{equation*}
  \cM_C \defeq \operatorname{Spec} H^{G_\cO}_*(\cR).
\end{equation*}
There is an algebra homomorphism $H_G^\ast(\mathrm{pt}) \rightarrow H_*^{G_\cO}(\cR)$, as in \cite[\S 3(vi)]{2016arXiv160103586B}. 

The algebra $H_*^{G_\cO}(\cR)$ is filtered in the same was as
\cite[Section 6(i)]{2016arXiv160103586B}, and as in
\cite[Proposition~6.8]{2016arXiv160103586B} we can prove that
$\cA$ is finitely generated.

The proof of the normality in
\cite[Proposition~6.12]{2016arXiv160103586B} was given by the
reduction to the cases when the gauge group is $\CC^\times$, $\SL(2)$
or $\PGL(2)$. That argument is applicable in our situation, and we are
reduced to the case a quiver with a single vertex. Then our
modification of the definition of the Coulomb branch is unnecessary and returns back to the original
situation. Therefore we see that $\cM_C$ is normal.

Finally, $\cM_C$ has a natural deformation quantization $\cA_\hbar$ defined below in \cref{sec:quant}.  This endows $\cM_C$ with a Poisson structure, which is generically symplectic as in \cite[Proposition 6.15]{2016arXiv160103586B}.  The subalgebra $H_G^\ast(\mathrm{pt})$ is Poisson commutative, and defines an integrable systems on $\cM_C$, see \cite[Section 1(iii)]{2016arXiv160103586B}.

\begin{Theorem}
$\cM_C$ is an irreducible normal variety of finite type.  It carries a Poisson structure which is generically symplectic, with an integrable system $\cM_C \rightarrow \operatorname{Spec} H_G^\ast(\mathrm{pt})$.
\end{Theorem}

\begin{Remark}
  As in \cite[Remark~3.9(3)]{2016arXiv160403625B}, one can also
  consider the $K$-theoretic Coulomb branch.
\end{Remark}

\begin{NB}
  This remark is added on May 10.
\end{NB}%

\section{Examples}\label{sec:example}

\subsection{}\label{subsec:1m}

We consider the case $I = \{1,2\}$, $c_{12} = -1$, $c_{21} = -m$
(where $m\in\ZZ_{>0}$), $\bw_{1}=\bw_2 = 0$, and $\bv_1 = \bv_2 = 1$. We choose the orientation $1\leftarrow 2$. Note that
$G$ is the two dimensional torus.
We consider the embedding
$\bz^*\colon H^{G_\cO}_*(\cR)\to H^{G_\cO}_*(\Gr_G)$ in
\cite[\S5(iv)]{2016arXiv160103586B}. Let $w_1$, $w_2$ be generators of
the equivariant cohomology ring of a point for the first and second factors of
$G = (\CC^\times)^2$. Let $\sfu_{a,b}$ denote the fundamental class of
the point $(a,b)\in\Gr_G = \ZZ^2$. We have 
\begin{equation*}
  H^{G_\cO}_*(\Gr_G) = \bigoplus_{a,b} \CC[w_1,w_2] \sfu_{a,b}
\end{equation*}
with $\sfu_{a,b} \sfu_{a',b'} = \sfu_{a+a',b+b'}$. Let $\sfy_{a,b}$
denote the fundametal class of the fiber of $\cR\to\Gr_G$ over
$(a,b)$. Then we have
\begin{equation*}
  \bz^*(w_1) = w_1, \quad \bz^*(w_2) = w_2, \quad
  \bz^*(\sfy_{a,b}) = (w_1-w_2)^{\max(b-ma,0)}\sfu_{a,b}.
\end{equation*}
\begin{NB}
    Nov.\ 26 : The formula was $\max(mb-a,0)$ in the earlier
    version. But it is not compatible with the twisted monopole formula.
\end{NB}%
Note that $\sfy_{1,m}=\sfu_{1,m}$ is invertible, with inverse $\sfy_{-1,-m}$. Therefore
$$H^{G_\cO}_*(\cR)\cong\CC[w_1, \sfy_{1,m}^\pm, \sfy_{0,1}, \sfy_{0,-1}]$$
Note for example, $w_1 - w_2 = \sfy_{0,1} \sfy_{0,-1}$. Thus the Coulomb
branch is $\cM_C = \BA^3\times\BA^\times$.
\begin{NB}
  This is an earlier version:
\begin{gather*}
  \bz^*(\sfy_{1,0}) = \sfu_{1,0}, \qquad
  \bz^*(\sfy_{k,1}) = (w_1 - w_2)^{m-k} \sfu_{k,1} \quad
  (k=0,1,\dots,m).
\end{gather*}
\begin{equation*}
  \begin{split}
    & \bz^*(\sfy_{k,1}) = (w_1 - w_2)^{m-k} \sfu_{k,1}
  = \bz^*(\sfy_{m-1,1})^{m-k} \bz^*(\sfy_{m,1})^{1+k-m}
  = \bz^*\left((\sfy_{m-1,1})^{m-k} (\sfy_{m,1})^{1+k-m}\right), \\
  & w_1 - w_2 = \sfu_{1,0} (w_1-w_2) \sfu_{m-1,1} (\sfu_{m,1})^{-1}
  = \bz^*(\sfy_{1,0} \sfy_{m-1,1} (\sfy_{m,1})^{-1}).
  \end{split}
\end{equation*}
Thus
$H^{G_\cO}_*(\cR)\cong\CC[w_1, (\sfy_{m,1})^\pm, \sfy_{1,0}, \sfy_{m-1,1}]$,
i.e., the Coulomb branch is $\BA^3\times\BA^\times$.
\end{NB}%

On the other hand, let us consider the folding of the Coulomb branch
of the quiver gauge theory $I = \{1,2_1,2_2,\dots, 2_m\}$ with edges
$1 \mbox{--} 2_j$ for all $j=1,\dots, m$ with $\bw_i = 0$, $\bv_i = 1$
for all $i\in I$. We consider the $\ZZ/m$-action on the quiver given
by $2_1\to 2_2\to \cdots \to 2_m\to 2_1$. See Figure~\ref{fig:G2} for
$m=3$. In order to distinguish the two groups for this theory and the former
gauge theory, let us write $\hat G = \prod \GL(V_i)$.
Note that the diagonal scalar $\CC^\times$ in $\hat G$
acts trivially on $\bN$, and so we have
$\hat G \cong \CC^\times\times (\CC^\times)^m$, $\bN = \CC^m$ where the
first $\CC^\times$ acts trivially on $\bN$, and $(\CC^\times)^m$ acts
on $\bN$ in the standard way. Therefore the (usual) Coulomb branch for $(\hat G, \bN)$ is
$\widehat{\cM}_C = \BA\times\BA^\times\times (\BA^2)^m$ and $\ZZ/m$ acts by cyclically permutating the
factors of $(\BA^2)^m$. Therefore the fixed point locus is also
$\BA\times\BA^\times\times \BA^2$.
Thus the former Coulomb branch is isomorphic to the fixed point locus
of the latter Coulomb branch:
\begin{Proposition}
For the above data, there is an isomorphism $\cM_C \cong (\widehat{\cM}_C)^{\ZZ/m}$.
\end{Proposition}
More concretely $w_1$, $w_2$ are identified with equivariant variables
for $\GL(V_1)$ and $\GL(V_{2_j})$, where the latter is independent of
$j$ on the $\ZZ/m$-fixed point locus. The function $\sfy_{a,b}$ is
identified with the restriction of the function
$\hat\sfy_{a,b_1,\dots,b_m}$ on $\widehat{\cM}_C$ given by the fundamental class over
$(a, b_1,\dots, b_m)\in \Gr_{\hat G}=\ZZ^{1+m}$ where
$b_1\ge\dots\ge b_m\ge b_1 - 1$ and $b=b_1+\dots+b_m$ (cf.\ the proof
of \cite[Prop.~4.1]{2016arXiv160403625B}).  
\begin{NB}
    $\sfy_{1,m} = \text{fundamental class over $(1,1,\dots,1)$}$,
    $\sfy_{0,1} = \text{fundamental class over $(0,1,0,\dots,0)$}$,
    $\sfy_{0,-1} = \text{fundamental class over $(0,0,\dots,0,-1)$}$.
\end{NB}%

\begin{NB}
  In \cite{bdf}, $d_i = \alpha_i\cdot \alpha_i$, where $\alpha_i$ is a
  simple \emph{coroot}. Then
  $d_i = \alpha_i\cdot\alpha_i/2\in\{ 1,2,3\}$, where
  $\alpha_i\cdot\alpha_i = 2$ for a short coroot. In \cite[5.7]{bdf},
  $d_i = 2$, $d_j=1$. The Dynkin diagram is
  $\circ_i \Leftarrow \circ_j$. Then $V_{\check\omega_i}$ is the
  natural four dimensional representation of $B_2 = \mathfrak{sp}(4)$,
  $V_{\check\omega_j}$ is the $5$-dimensional representation contained
  in $\Wedge^2 V_{\check\omega_j}$. So $i$ is $2$, $j$ is $1$.
\end{NB}%

In the cases $m=2, 3$, we can identify our modified Coulomb branch with an \emph{open zastava space} $\oZ^\alpha$.  Recall that $\oZ^\alpha$ is the moduli space of based maps from $\mathbb{P}^1$ to the flag variety, of degree $\alpha$, see \cite[\S 2]{bdf}, \cite[\S 2(i)]{2016arXiv160403625B}.

\begin{Lemma}
\label{lemma: open zastava}
For $m=2$ (resp.~$m=3$), $\cM_C$ is isomorphic to the open zastava space $\oZ^{\alpha_1+\alpha_2}$ of type $B_2$ (resp.~type $G_2$).
\end{Lemma}
\begin{proof}
Explicitly, we identify our description for $m=2$ with the $B_2$ type open zastava
space \cite[\S5.7]{bdf}\footnote{$d_i$ in \cite{bdf} is the square
  length of the simple \emph{coroot} for $i$, while it is of the
  simple \emph{root} here. Therefore $i$ (resp. $j$) in \cite{bdf} is
  $2$ (resp.\ $1$) here.
  \begin{NB}
    The footnote is added on May 10.
  \end{NB}%
} by $w_1 = -A_2$, $w_2 = -A_1$,
$\sfy_{1,2} = b_{03}$, $\sfy_{0,1} = b_{01}$,
and $\sfy_{0,-1} = b_{02} b_{03}^{-1}$, noticing that $b_{03}$ is
invertible.  
\begin{NB}
    $\sfy_{1,1} = \sfy_{1,2} \sfy_{0,-1} = b_{02}$,
    $b_{12} = b_{02}^2 b_{03}^{-1} = \sfy_{1,1}^2 \sfy_{1,2}^{-1} =
    \sfy_{1,0}$,
    $b_{02}(A_1 - A_2) = \sfy_{1,1}(w_1-w_2) = \sfy_{0,1} \sfy_{1,0} =
    b_{01} b_{12}$,
    $b_{03}(A_1 - A_2) = \sfy_{1,2}(w_1-w_2) = \sfy_{0,1} \sfy_{1,1} =
    b_{01} b_{02}$.

  Note $y_i= b_{01} = \sfy_{0,1}$,
  $y_j = b_{12} = \sfy_{1,0} = (\sfy_{0,-1})^2 \sfy_{1,2}$.
\end{NB}%

Similarly, in the $m=3$ case we appeal to the description of the $G_2$ type open zastava in terms of coordiantes, from \cref{sec:app}. 
\end{proof}

Another perpsective on this result is via folding.  Recall that there is an \'etale rational coordinate system
$(y_{i,r}, w_{i,r})_{i\in I, 1\le r\le \bv_i}$ on the open zastava
space $\oZ^\alpha$ for finite type \cite{fkmm,bdf}. We claim that it is
compatible with  folding of the same type described above, namely the coordinate system for
$B_2$, $G_2$ is the restriction of the coordinate system for $A_3$,
$D_4$ to the $\ZZ/m$-fixed point ($m=2,3$), respectively.
For $B_2$ with the above choice of $\bv$, this can be checked directly from
\cite[\S5.7]{bdf} as $y_i = b_{01} = \sfy_{0,1}$,
$y_j = b_{12} = b_{02}^2 b_{03}^{-1} = \sfy_{1,0}$.
In general, it is enough to check the assertion when $\bv_i$ is $1$
dimensional for a single vertex $i$ and $0$ otherwise by the
compatibility of the coordinate system and the factorization in
\cite[Th.~1.6(3)]{bdf}, as the factorization and folding are
compatible. 
In that case, a based map factors through $\CP^1$ via the
embedding of $\CP^1$ into the flag variety corresponding to the vertex
$i$. Then the assertion is clear.
Alternatively we use the description from \cref{sec:app} for $G_2$ to
argue as in the $B_2$ case.

Recall that the isomorphism between $\oZ^\alpha$ and the corresponding Coulomb
branch was defined so that the coordinate system $(y_{i,r}, w_{i,r})$
is mapped to $(\sfy_{i,r},w_{i,r})$, where the latter $w_{i,r}$ is an
equivariant variable as above, and $\sfy_{i,r}$ is the fundamental
class of the fiber over the point corresponding to $w_{i,r}$
\cite[\S3]{2016arXiv160403625B}. Since the coordinate system is
compatible with the above folding, and $\sfy_{1,0}$, $\sfy_{0,1}$ for
$B_2$, $G_2$ are restriction of appropriate $\hat\sfy_{1,0,\dots,0}$,
$\hat\sfy_{0,1,0,\dots,0}$ for $A_3$, $D_4$, the coordinate system
$(y_{1,0}, y_{0,1}, w_1, w_2)$ for $B_2$, $G_2$ is identified with
$(\sfy_{1,0}, \sfy_{0,1}, w_1, w_2)$.

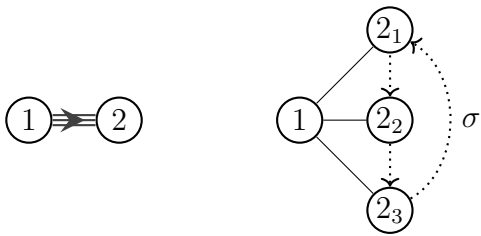
\begin{figure}[htbp]
    \centering
\begin{tikzpicture}[scale=1.2,
circled/.style={circle,draw
,thick,inner sep=0pt,minimum size=6mm},
squared/.style={rectangle,draw
,thick,inner sep=0pt,minimum size=6mm},
triplearrow/.style={
  draw=black!75,
  color=black!75,
  thick,
  double distance=3pt, 
  decoration={markings,mark=at position .75 with {\arrow[scale=.7]{>}}},
  postaction={decorate},
  >=stealth}, 
thirdline/.style={draw=black!75, color=black!75, thick, -
}
]
\node[circled] (v1) at ( 1,0)  {$1$};
\node[circled] (v2) at ( 2,0)  {$2$};
\draw[triplearrow] (v1.east) -- (v2.west);
\draw[thirdline] (v1.east) -- (v2.west);

\node[circled] (vv1) at ( 4,0)  {$1$};
\node[circled] (vv2) at ( 5,1)  {$2_1$};
\node[circled] (vv4) at ( 5,-1)  {$2_3$}
   edge [->,bend right=60,dotted,thick] node[auto,swap] {$\sigma$} (vv2);
\node[circled] (vv3) at ( 5,0)  {$2_2$}
  edge [<-,dotted,thick] (vv2)
  edge [->,dotted,thick] (vv4);
\draw [-] (vv1) -- (vv2);
\draw [-] (vv1) -- (vv3);
\draw [-] (vv1) -- (vv4);
\end{tikzpicture}
\caption{$G_2$ and the folding of $D_4$}
    \label{fig:G2}
\end{figure}

\begin{NB}
  The following two subsections are added on Dec.~28.
\end{NB}%

%

\begin{NB2}
July 12: Commented out the subsection here on general rank 2 cases.  Moved to Appendix C instead
\end{NB2}

\subsection{}
\label{subsec:1mzastava}

As in \cite[\S3(ii)]{2016arXiv160103586B}, we may consider the positive part
$\Gr^+_G$ of the affine Grassmannian $\Gr_G$: the
subvariety consisting of $(\mathcal E_i,\varphi_i)$ such that
$\varphi_i$ extends through the puncture as an embedding
$\mathcal E_i\hookrightarrow V_i\otimes_\CC\shfO_{D_i}$.
We define $\cR^+ \subset \cR$ as the preimage.
Then $H^{G_\cO}_*(\cR^+)$ forms a convolution subalgebra of
$H^{G_\cO}_*(\cR)$, equipped with an algebra homomorphism
$H^{G_\cO}_*(\cR^+)\to H^*_{G}(\mathrm{pt})$.

Let us consider $H^{G_\cO}_*(\cR^+)$ for the example in
\cref{subsec:1m}. It is the subalgebra generated by $w_1$, $w_2$,
$\sfy_{a,b}$ with $a$, $b\ge 0$. It is easy to check that it is, in
fact, generated by $w_1$, $w_2$, $\sfy_{0,1}$, $\sfy_{1,0}$,
$\sfy_{1,1}$, \dots, $\sfy_{1,m}$.
\begin{NB}
  Note $\sfy_{0,1} = (w_1 - w_2) \sfu_{0,1}$,
  $\sfy_{1,0} = \sfu_{1,0}$, $\sfy_{1,1} = \sfu_{1,1}$, \dots,
  $\sfy_{1,m} = \sfu_{1,m}$. If $(a,b)$ is above the line $y = mx$, we
  have $\sfy_{a,b} = \sfy_{a,ma} \sfy_{0,b-ma}$. If it is below the
  line, we write $b = b_1 + b_2 + \cdots + b_a$ with $0\le b_i\le m$,
  then $\sfy_{a,b} = \sfy_{1,b_1} \sfy_{1,b_2}\cdots \sfy_{1,b_a}$.
\end{NB}%
We have $\sfy_{1,b} \sfy_{0,1} = (w_1 - w_2) \sfy_{1,b+1}$ for
$0\le b\le m-1$, and
$\sfy_{1,b_1}\sfy_{1,b_2}\cdots = \sfy_{1,b_1'} \sfy_{1,b_2'}\cdots$ 
for $0\le b_i, b'_i\le m$ with $b_1+b_2 +\cdots = b_1' + b_2' +\cdots$.

If $m=2$, the only nontrivial relation of the latter type is
$\sfy_{1,0}\sfy_{1,2} = \sfy_{1,1}^2$.
\begin{NB}
  The next possible candidate is for $b_1+b_2+b_3=3$, but it is
  $\sfy_{1,0} \sfy_{1,1} \sfy_{1,2} = \sfy_{1,1}^3$, and can be
  deduced from the above. Next one is
  $\sfy_{1,1}^2 \sfy_{1,2} = \sfy_{1,0} \sfy_{1,2}^2$, which is also a
  consequence of the above.
\end{NB}%
This coincides with the presentation of the $B_2$ type zastava space
\cite[\S5.7]{bdf} by $w_1 = -A_2$, $w_2 = -A_1$,
$\sfy_{0,1} = b_{01}$, $\sfy_{1,0} = b_{12}$, $\sfy_{1,1} = b_{02}$,
$\sfy_{1,2} = b_{03}$.

If $m=3$, we have two 
\begin{NB}
    corrected on Mar.~6.
\end{NB}%
more relations
$\sfy_{1,0} \sfy_{1,3} = \sfy_{1,1}\sfy_{1,2}$,
$\sfy_{1,1}\sfy_{1,3} = \sfy_{1,2}^2$.  We cannot find this presentation of the zastava space for $G_2$ for
degree $\alpha_1+\alpha_2$ in the literature. Therefore we include the proof in
the appendix~\ref{sec:app}.
\begin{NB}
  $\sfy_{1,0}^2 \sfy_{1,3} = \sfy_{1,1}^3$ can be deduced.
  But no other relations ?

  \cite[\S5.8]{bdf} does not deal with the presentation of the $G_2$
  type zastava.
\end{NB}%

Together, we obtain:
\begin{Lemma}
\label{lemma: zastava}
For $m=2$ (resp.~$m=3$), $\operatorname{Spec}H^{G_\cO}_*(\cR^+)$ is isomorphic to the zastava space $Z^{\alpha_1+\alpha_2}$ of type $B_2$ (resp.~type $G_2$).
\end{Lemma}

\begin{NB2}%
For $G_2$, I believe a complete set of relations of the latter type is 
$$
\sfy_{1,1}^2 = \sfy_{1,0} \sfy_{1,2}, \ \ \sfy_{1,1} \sfy_{1,2} = \sfy_{1,0}\sfy_{1,3}, \ \ \sfy_{1,2}^2 = \sfy_{1,1} \sfy_{1,3} 
$$
\end{NB2}%
\begin{NB}
    I take the following from {\bf NB}.
\end{NB}%
\begin{Remark}
For general $m$, a complete set of relations of the latter type are as follows: for all $1\leq a\leq b < m$,
$$
\sfy_{1,a} \sfy_{1,b} = \left\{ \begin{array}{cl} \sfy_{1,0}\sfy_{1, a+b}, & \text{if } a+b \leq m, \\  \sfy_{1, a+b-m}\sfy_{1, m}, & \text{if } a+b>m \end{array} \right.
$$
\end{Remark}
\begin{NB2}
Namely, we know the subalgebra generated by the elements $\sfy_{1, b}$ with $0\leq b \leq m$ has a basis $\sfy_{a,b}$ where $ma \geq b$.  If we write $b = m q + r$ with $0\leq r<m$, then 
$$
\sfy_{a,b} = \sfy_{1,m}^q \sfy_{1,0}^{a- q - \lceil\tfrac{r}{m} \rceil } \prod_{b=1}^{m-1} \sfy_{1,b}^{\delta_{r, b}}
$$
In particular at most one $\sfy_{1,b}$ with $1\leq b < m$ appears, with multiplicity $\leq 1$. The above relations let us reduce any product of the elements $\sfy_{1, b}$ with $0\leq b \leq m$ to this form, as each such relation reduces the total degree in the elements $\sfy_{1,b}$ with $1\leq b < m$ by one.
\end{NB2}%

\begin{NB}
    I added the following on Mar.~6.
\end{NB}%

\section{Slices}\label{sec:slice}

Consider an adjoint group $\mathscr{G}$ of  $BCFG$ type, with fundamental coweights $\{ \Lambda_i\}$ and simple coroots $\{\alpha_i\}$. Given a dominant coweight $\lambda$ for $\mathscr{G}$, and a coweight $\mu$ such that $\lambda \geq \mu$, we define the corresponding \emph{generalized affine Grassmannian slice} $\overline{\cW}{}^\lambda_\mu$ as in \cite[\S 2(ii)]{2016arXiv160403625B}. Recall that in the case when $\mu$ is itself dominant, $\overline{\cW}{}^\lambda_\mu$ is isomorphic to an ordinary affine Grassmannian slice in $\Gr_{\mathscr{G}}$ as defined in \cite[\S2]{bf14}, \cite[\S 2B]{kwy}.

The proofs of properties of $\overline{\cW}{}^\lambda_\mu$, given in \cite[\S 2]{2016arXiv160403625B}, work for non simply-laced types.  In particular, $\overline{\cW}{}^\lambda_\mu$ is Cohen-Macaulay, normal, and affine.   It has an integrable system $\overline{\cW}{}^\lambda_\mu \rightarrow \mathbb{A}^\alpha$ where $\alpha = \lambda - \mu$, which satisfies factorization as in \cite[\S 2(ix)]{2016arXiv160403625B}.

Thanks to the analysis in the previous section, we can apply the argument in \cite[\S3]{2016arXiv160403625B} to symmetrizable cases:

\begin{Theorem}\label{thm:slice}
    Suppose that the valued quiver is of type $BCFG$, with adjoint group ${\mathscr{G}}$ as above. Then
    
    \textup{(1)} Suppose $W=0$. Then $\cM_C = \Spec H^{G_\cO}_*(\cR)$
    is isomorphic to the open zastava space $\oZ^\alpha$ for ${\mathscr{G}}$ of degree $\alpha = \sum_i \dim V_i\cdot\alpha_i$.

    \textup{(2)} Suppose $W\neq 0$. Then
    $\cM_C = \Spec H^{G_\cO}_*(\cR)$ is isomorphic to the generalized slice
    $\overline{\cW}{}^\lambda_\mu$ for ${\mathscr{G}}$ where
    $\lambda$, $\mu$ are given by
    $\lambda = \sum_i \dim W_i\cdot\Lambda_i$,
    $\mu = \lambda - \sum_i \dim V_i\cdot\alpha_i$.

    \textup{(3)} $\Spec H^{G_\cO}_*(\cR^+)$ is isomorphic to the
    zastava space $Z^\alpha$ for ${\mathscr{G}}$ of degree 
    $\alpha = \sum_i \dim V_i\cdot \alpha_i$.
\end{Theorem}
\begin{proof}
Since the proofs are essentially the same as in \cite[\S3]{2016arXiv160403625B}, we simply indicate the differences.  In both parts we wish to appeal to \cite[Thm.~5.26]{2016arXiv160103586B}: in a certain precise sense, it suffices to identify the varieties in codimension 1.  This result generalizes to our present setting, with the same proof.

For (1),(2) we follow the proof of \cite[Thm.~3.1, 3.10]{2016arXiv160403625B}. Using the same notation, the only difference comes when comparing the varieties in the case when $t$ lies on a diagonal divisor $(w_{i,r} - w_{j,s})(t) = 0$ where $i\neq j$.  In our present BCFG setting, we may meet factors of open zastava's $\oZ^{\alpha_1+\alpha_2}$ of type $B_2, G_2$ in addition to the usual $A_1 \times A_1$ and $A_2$ types already discussed in \cite[Rem.~2.2]{2016arXiv160403625B}.  In these new cases we apply Lemma \ref{lemma: open zastava} to complete the proof.

For (3) we follow  \cite[Remark~3.15]{2016arXiv160403625B}, this time making use of Lemma \ref{lemma: zastava}.
\end{proof}

\begin{Remark}
Part (2) extends to relate the \emph{flavor symmetry deformation} of $\cM_C$ with a BD slice, generalizing \cite[Thm.~3.20]{2016arXiv160403625B}. The same applies for \cite[Thm.~5.5]{2018arXiv180511826B}.
\end{Remark}

\begin{NB}
  Zastava ? We do not have a presentation of zastava for $G_2$. This
  could be a problem.... I add (2).
\end{NB}%

\begin{NB}
    I added (2) and the above sentence on Mar.~6.
\end{NB}%

\section{Quantization}\label{sec:quant}

In this section, we connect the deformed algebra $H^{G_\cO \rtimes \CC^\times}_*(\cR)$ with truncated shifted Yangians in type BCFG, extending the results of \cite[Appendix B]{2016arXiv160403625B}.

\subsection{Loop rotation}
\label{section: loop rotation}

To discuss the deformation $H^{G_\cO\rtimes \CC^\times}_*(\cR)$, we must first make a choice of $\CC^\times$--action on $\cR$.  
\begin{NB}
  Earlier wrong version:

  Let $r^\vee$ be the least common multiple of $d_i$ ($i\in I$). (It is
equal to the lacing number for $BCFG$.) We define a
$\CC^\times$-action on $\CC[[z_i]]$ by
\begin{equation*}
   z_i \mapsto z_i \tau^{r^\vee/d_i} \qquad (\tau\in\CC^\times).
\end{equation*}
Then the equation $z_i^{f_{ji}} = z_j^{f_{ij}}$ is preserved, as
$d_i f_{ij} = d_j f_{ji}$. Therefore we have an induced
$\CC^\times$--action on $\cR$.
\end{NB}%
\begin{NB2}
I have followed these new conventions below.
\end{NB2}
This action will depend on a choice of symmetrizers $(d_i)\in \ZZ_{>0}^I$ for our Cartan matrix $(c_{ij})_{i,j\in I}$.  We define a $\CC^\times$--action on $\CC[[z_i]]$ by
\begin{equation*}
   z_i \mapsto z_i \tau^{d_i} \qquad (\tau\in\CC^\times).
\end{equation*}
Then the equation $z_i^{f_{ij}} = z_j^{f_{ji}}$ is preserved, as
$d_i f_{ij} = d_j f_{ji}$. Therefore we have an induced
$\CC^\times$--action on $\cR$.

\begin{NB2}
June 27: Added (perhaps too obvious) comment above that here the choice of $d_i$ matters.
\end{NB2}

\subsection{Embedding into the ring of difference operators}
\label{sec: embedding into the ring of difference operators}

Consider a valued quiver along with vector spaces $V = \bigoplus V_i$ and $W = \bigoplus W_i$ as above.  Consider the deformed algebra
$$
\cA_\hbar := H_\ast^{\widetilde{G}_\cO \rtimes \CC^\times} (\cR),
$$
where $\widetilde{G} = G \times T(W)$ with $T(W) \subset \prod_i \GL(W_i)$ the standard maximal torus, and where the $\CC^\times$--action on $\widetilde{G}_\cO$ and $\cR$ is induced by its action on $\CC[[z_i]]$ as in the previous section.  We choose a basis $\equi_1,\ldots,\equi_N$ of the character lattice of $T(W)$ compatible with the product decomposition $T(W) = \prod_i T(W_i)$.  Thus $\cA_\hbar$ is naturally an algebra over
$$
H_{T(W)\times \CC^\times}^\ast(\mathrm{pt}) = \CC[\hbar,\equi_1,\ldots,\equi_N],
$$
which is a central subalgebra (see \cite[Section 3(viii)]{2016arXiv160103586B}).

As in \cite[Appendix A(i)--A(ii)]{2016arXiv160403625B}, we can construct an embedding 
$$
\bz^*(\iota_*)^{-1}\colon \cA_\hbar \hookrightarrow \widetilde{\cA}_\hbar 
$$
where we define an algebra
$$
\widetilde{\cA}_\hbar:= \CC[\hbar, \equi_1,\ldots, \equi_N]\big\langle w_{i,r}, \sfu_{i,r}^{\pm 1}, \hbar^{-1}, (w_{i,r}-w_{i,s} + m d_i \hbar)^{-1} : i \in Q_0, 1\leq r\neq s \leq \bv_i, m \in \ZZ\big\rangle
$$
by the relations  $[\sfu_{i,r}^{\pm 1}, w_{j, s} ] = \pm \delta_{i,j} \delta_{r, s} \hbar d_i \sfu_{i,r}^{\pm 1 }$ (all other elements commute).  Note that $\widetilde{\cA}_{\hbar}$ is a localization of $H_\ast^{T \times T(W) \times \CC^\times}(\Gr_T)$.

For the homology classes of $\cR$ associated to preimages $\cR_\lambda$ of closed $G_\cO$--orbits, we can explicitly write down the image under the map $\bz^{-1}(\iota_\ast)^{-1}$, following \cite[Proposition A.2]{2016arXiv160403625B}.  Let $\lambda$ be a miniscule dominant coweight, $W_\lambda \subset W$ its stabilizer, and $f \in \CC[\mathfrak{t}]^{W_{\lambda}}$.  Then 
$$
\bz^\ast (\iota_\ast)^{-1} f [ \cR_\lambda] = \sum_{\lambda' = w \lambda \in W \lambda} \frac{wf \times e_{\lambda'}}{e( T_{\lambda'} \Gr_G^{\lambda})} \sfu_{\lambda'}
$$
where $e_{\lambda'}$ denotes the Euler class of the fiber of $\cT$ over $\lambda'$ modulo the fiber of $\cR$ over $\lambda'$. 

Following \cite[Section A(ii)]{2016arXiv160403625B}, we will compute these classes for the cocharacters $\varpi_{i,n}$ and $\varpi_{i,n}^\ast$ of $GL(V_i)$. We find

\begin{NB2}
I have omitted edge loops below.  Should we include them?

I think that it is fine, but we should mention that we omit edge loops
for simplicity.
\end{NB2}%
\begin{equation}
\label{eq: monopole calculation 1}
\bz^\ast (\iota_\ast)^{-1} f(\uQ_i) [ \cR_{\varpi_{i,n}} ]  =
\end{equation}
$$
 \sum_{\substack{ I  \subset \{1,\ldots,\bv_i\}\\  \# I = n}} f(w_{i, I} ) \frac{\displaystyle\prod_{\substack{ h \in Q_1: o(h) = i\\ r\in I}} \prod_{s=1}^{\bv_{i(h)}}\prod_{p=0}^{f_{i(h),i}-1}\big(- w_{i,r} +w_{i(h), s}  + (-d_i f_{i, i(h)} + p d_{i(h)}\big) \hbar)^{g_{i, i(h)} } }{\displaystyle\prod_{r\in I, s\notin I} (w_{i,r} - w_{i,s})} \prod_{r\in I} \sfu_{i,r} 
$$
\begin{NB2}
Using that 
$$
\Hom_{\CC[[z]]} ( z_i \CC[[z_i]], \CC[[z_j]]) / \Hom_{\CC[[z]]}( \CC[[z_i]], \CC[[z_j]] )
$$
has basis over $\CC$ given by the (classes of the) $\CC[[z]]$--linear maps defined by $z_i^{d_i f_{ij}} \mapsto z_j^p$, for $0\leq p < f_{ji}-1$, which have degrees $-d_i f_{ij} + p d_j$.
\end{NB2}%
and
{\allowdisplaybreaks
\begin{gather}
\label{eq: monopole calculation 2}
\bz^\ast(\iota_\ast)^{-1} f(\uS_i) [ \cR_{\varpi_{i,n}^\ast}] =
\\
\begin{aligned}[t]
  & \sum_{\substack{I \subset \{1,\ldots,\bv_i\} \\ \# I = n}} f(w_{i,I} - d_i \hbar) \prod_{\substack{r\in I \\ k : i_k = i}} (w_{i,r} - \equi_k - d_i \hbar)
  \\
  & \qquad
  \times\frac{\displaystyle\prod_{\substack{h \in Q_1 : i(h) = i \\ r\in I }} \prod_{s=1}^{\bv_{o(h)}} \prod_{p=0}^{f_{o(h),i}-1}\big( w_{i,r} - w_{o(h), s}- (d_i+p d_{o(h)}) \hbar\big)^{g_{i, o(h)} }}{\displaystyle\prod_{r\in I, s\notin I} (-w_{i,r} + w_{i,s})} \prod_{r\in I} \sfu_{i,r}^{-1}. \notag
\end{aligned}
  \end{gather}
}
\begin{NB2}
Using that
$$
\Hom_{\CC[[z]]} ( \CC[[z_j]], z_i^{-1} \CC[[z_i]]) / \Hom_{\CC[[z]]}( \CC[[z_j]], \CC[[z_i]] )
$$
has basis over $\CC$ of the maps defined by $z_j^p \mapsto z_i^{-1}$, for $0\leq p< f_{ji}-1$, which have degrees $-d_i - p d_j$.
\end{NB2}%
\begin{NB}
  It seems that the summation should run from $p=0$, not $p=1$.
\end{NB}%

\begin{NB2}
You are right; I've changed the summation to run from $p=0$..
\end{NB2}

\subsection{Shifted Yangians}  
The definition of shifted Yangians given in \cite[Definition B.2]{2016arXiv160403625B} extends naturally to all finite types. Thus in BCFG types, for any coweight $\mu$ there is a corresponding shifted Yangian $Y_\mu$. It is a $\CC$--algebra, with generators $E_i^{(q)}, F_i^{(q)}, H_i^{(p)}$ for $i \in Q_0$, $q>0$ and $p>-\langle \alpha_i^\vee, \mu \rangle$. Here $\alpha_i^\vee$ denotes the simple {\em root} for $i\in I$. \begin{NB2} Include relations? \end{NB2}  \begin{NB2} July 5: Corrected notation from previous $p> - \mu_i$ \end{NB2}

The properties of $Y_\mu$ established in \cite{fkprw} have straightforward extensions to all finite types.  In particular, $Y_\mu$ has a PBW basis, and for any coweights $\mu_1,\mu_2$ with $\mu = \mu_1+\mu_2$ there is a filtration $F_{\mu_1,\mu_2}^\bullet Y_\mu$ of $Y_\mu$.  The associated graded $\operatorname{gr}^{F_{\mu_1,\mu_2}} Y_\mu$ is commutative, and the Rees algebras $\operatorname{Rees}^{F_{\mu_1,\mu_2}} Y_\mu$ are all canonically isomorphic as algebras (although not as graded algebras).  For the purposes of this paper, we will choose $\mu_1,\mu_2$ as follows: 
$$
\langle \mu_1, \alpha_i^\vee \rangle = \langle \lambda, \alpha_i^\vee\rangle - \bv_i + \sum_{h: i(h) = i} \bv_{o(h)} c_{o(h), i}, \quad \langle \mu_2, \alpha_i^\vee \rangle = -\bv_i + \sum_{h: o(h) = i} \bv_{i(h)} c_{i(h), i}
$$
We write $\mathbf{Y}_\mu := \operatorname{Rees}^{F_{\mu_1,\mu_2}} Y_\mu$ for the corresponding Rees algebra, which we view as a graded algebra over $\CC[\hbar]$ with $\deg \hbar =1$.

Below, we work with the larger algebra $Y_\mu[\equi_1,\ldots,\equi_N] = Y_\mu \otimes_\CC \CC[\equi_1,\ldots, \equi_N]$, where $N = \sum_i \bw_i$.  The filtration $F_{\mu_1,\mu_2}$ extends to $Y_\mu[\equi_1,\ldots,\equi_N]$ by placing all $\equi_i$ in degree 1.  We denote the corresponding Rees algebra by $\mathbf{Y}_\mu[\equi_1,\ldots, \equi_N]$.

Denote
$$
\Equi_i(\formal) = \prod_{k: i_k = i} (\formal - \equi_k - d_i \hbar),
$$
and define elements $A_i^{(s)} \in Y_\mu[\equi_1,\ldots, \equi_N]$ for $s>0$ according to 
\begin{equation}
H_i(\formal) = \Equi_i(\formal) \frac{\prod_{j\neq i} \prod_{p = 1}^{-c_{ji}} \big(\formal - \tfrac{1}{2}d_i c_{ij} - p d_j \big)^{\bv_j}}{\formal^{\bv_i} (\formal-d_i)^{\bv_i}} \frac{ \prod_{j\neq i} \prod_{p=1}^{-c_{ji}} A_j(\formal-\tfrac{1}{2}d_i c_{ij}-pd_j )}{A_i(\formal) A_i(\formal-d_i ) }
\end{equation}
where
$$
H_i(\formal) =  \formal^{\mu_i} + \sum_{r> - \mu_i} H_i^{(r)} \formal^{-r}, \quad A_i(\formal) = 1 + \sum_{s >0} A_i^{(s)} \formal^{-p}
$$
\begin{NB2} Alternatively, we may modify the shifts that appear here, c.f.~discussion around Theorem \ref{Theorem: shifted Yangian to Coulomb branch}.  In particular, we can replace
$$
A_i(z) \mapsto \left( \frac{z+\sigma_i\hbar}{z} \right)^{\bv_i} A_i(z+\sigma_i\hbar)
$$
for any shifts $\sigma_i \in \CC$.  Note that the ideal generated by the elements $A_i^{(r)}$ with $r>m_i$ remains the same.
\end{NB2}

\subsection{A representation using difference operators}
Recall the $\CC[\hbar,\equi_1,\ldots,\equi_N]$--algebra $\widetilde{\cA}_\hbar$ defined in \cref{sec: embedding into the ring of difference operators}.
This algebra has a grading, defined by $\deg \hbar = \deg \equi_k = \deg w_{i,r} =1$ and $\deg \sfu_{i,r}^{\pm 1} = 0$.   Denote
$$
W_i(\formal) = \prod_{s=1}^{\bv_i} (\formal- w_{i,s}), \qquad  W_{i,r}(\formal) = \prod_{\substack{s=1 \\ s\neq r}}^{\bv_i} (\formal - w_{i,s})
$$

The following result is a common generalization of \cite[Corollary B.17]{2016arXiv160403625B} and \cite[Theorem 4.5]{kwy}, which were in turn generalizations of work of Gerasimov-Kharchev-Lebedev-Oblezin \cite{GKLO}.
\begin{Theorem}
\label{thm: GKLO}
There is a homomorphism of graded $\CC[\hbar,\equi_1,\ldots,\equi_N]$--algebras 
\begin{equation*}
\Phi_\mu^\lambda : \mathbf{Y}_\mu[\equi_1,\ldots,\equi_N] \longrightarrow \widetilde{\cA}_\hbar,
\end{equation*}
defined by
\begin{align*}
A_i(\formal) & \mapsto \formal^{-\bv_i} W_i(\formal), \\
E_i(\formal) & \mapsto - d_i^{1/2} \sum_{r=1}^{\bv_i} \Equi_i(w_{i,r}) \frac{\displaystyle \prod_{ h\in Q_1 : i(h) = i} \prod_{p=1}^{-c_{o(h), i}} W_{o(h)} \big( w_{i,r} - (\tfrac{1}{2}d_i c_{i, o(h)} + p d_{o(h)} ) \hbar\big) }{(\formal-w_{i,r}) W_{i,r}(w_{i,r})} \sfu_{i,r}^{-1}, \\
F_i(\formal) & \mapsto d_i^{1/2} \sum_{r=1}^{\bv_i}  \frac{\displaystyle \prod_{h\in Q_1: o(h) = i} \prod_{p=1}^{-c_{i(h), i}} W_{i(h)}\big( w_{i,r} - (\tfrac{1}{2} d_i c_{i, i(h)} - d_i + p d_{i(h)} ) \hbar \big) }{(\formal-w_{i,r} -d_i \hbar) W_{i,r}(w_{i,r})} \sfu_{i,r}
\end{align*}
\end{Theorem}
In simply-laced type, a proof of this theorem was given in \cite[\S B(iii)--B(vii)]{2016arXiv160403625B}.  In all finite types, a geneneralization of this theorem for shifted quantum affine algebras was proven in \cite{ft}.  We thus omit the proof.

\subsection{Relation to the quantized Coulomb branch}
\label{subsection: Relation to the quantized Coulomb branch}

Consider the setup of \cref{sec: embedding into the ring of difference operators}, restricted to BCFG type.  Recall that in this case $g_{ij} =1$ and thus $f_{ij} = |c_{ij}|$, whenever $c_{ij}<0$.  With this in mind, we see that the right-hand sides of equations (\ref{eq: monopole calculation 1}), (\ref{eq: monopole calculation 2}) for $n=1$ are nearly identical to the images $\Phi_\mu^\lambda(F_i^{(r)}), \Phi_\mu^\lambda(E_i^{(r)})$ from the previous theorem, modulo shifts by $\hbar$ in their respective numerators.

Choose $\sigma_i\in \ZZ$ for each $i\in Q_0$, which solve the following system of equations: for each $h\in Q_1$, we require that
\begin{equation}
\label{eq: shifts sigma}
\tfrac{1}{2} d_{o(h)} c_{o(h), i(h)} = \sigma_{o(h)} - \sigma_{i(h)} - d_{o(h)} + d_{i(h)}
\end{equation}
Since $(Q_0,Q_1)$ is an orientation of a tree, a solution exists and is unique up to an overall additive shift.  However, in general these equations depend upon the choice of orientation of the Dynkin diagram.

\begin{Theorem}
\label{Theorem: shifted Yangian to Coulomb branch}
Fix integers $\sigma_i$ satisfying (\ref{eq: shifts sigma}). Then there is a unique graded $\CC[\hbar,\equi_1,\ldots, \equi_N]$--algebra homomorphism
$$
\overline{\Phi}_\mu^\lambda: \mathbf{Y}_\mu[\equi_1,\ldots,\equi_N] \longrightarrow \cA_\hbar
$$	
such that
\begin{align*}
A_i^{(r)} & \mapsto (-1)^p e_r\big( \{ w_{i,r} - \sigma_i \hbar \} \big), \\
E_i^{(r)} & \mapsto (-1)^{\bv_i} d_i^{-1/2} \big(c_1(\uS_i)+ (d_i -\sigma_i) \hbar\big)^{r-1} \cap [ \cR_{\varpi_{i,1}^\ast}], \\
F_i^{(r)} & \mapsto (-1)^{\sum_{h: o(h) = i} a_{i(h),i} \bv_{i(h)}} d_i^{-1/2} \big(c_1(\uQ_i) + (d_i-\sigma_i) \hbar\big)^{r-1} \cap [ \cR_{\varpi_{i,1}} ]
\end{align*}
\end{Theorem}
\begin{NB2}
This map is actually degree-doubling
\end{NB2}
\begin{NB2}
Dec 30: Added factors of $d_i^{-1/2}$
\end{NB2}

\begin{Remark}
The integers $\sigma_i$ play the role of a ``shift'' in the action of the loop rotation from \cite[Section 2(i)]{2016arXiv160103586B}, where the loop $\CC^\times$ also acts on $\bN$ by weight $1/2$.  Indeed, in our present setting we could modify the loop action of $\CC^\times$ from \cref{section: loop rotation}, so that it also scales $V_i, W_i$ with weight $\sigma_i$. (Thus when acting on $\cR$, in addition to rotating the discs $D_i$, $\tau \in \CC^\times$ scales the morphism $s_{ij}$ by $\tau^{\sigma_i - \sigma_j}$, and scales $s_i$ by 1). With this modified action, no shifts by $\sigma_i$ would be needed in the statement of the theorem.  Note that since this modified $\CC^\times$--action factors through the usual action of $G\times \CC^\times$, the modified algebra is isomorphic to the original (c.f.~\cite[Remark 3.24(2)]{2016arXiv160103586B}).
\end{Remark}

\begin{proof}[Proof of Theorem \ref{Theorem: shifted Yangian to Coulomb branch}]
We may argue using the previous remark, and modify the loop $\CC^\times$--action while preserving the algebra $\cA_\hbar$ up to isomorphism.  We give an equivalent elementary argument:

Consider the automorphism $\sigma$ of $\widetilde{\cA}_\hbar$ defined by $w_{i,r} \mapsto w_{i,r} + \sigma_i \hbar$ and $\equi_k \mapsto \equi_k + \sigma_{i_k} \hbar$, while fixing the generators $\hbar, \sfu_{i,r}^{\pm 1}$. We claim that in $\widetilde{\cA}_\hbar$ we have equalities
$$
\Phi_\mu^\lambda( x ) = \sigma \circ \bz^\ast(\iota_\ast)^{-1} (y),
$$
where $x\in \{A_i^{(r)}, E_i^{(r)}, F_i^{(r)}\}$, and where $y\in \cA_\hbar$ is the claimed image $\overline{\Phi}_\mu^\lambda(x)$ from the statement of the theorem.  For the elements $x=A_i^{(r)}$ this is obvious.  For  $x=E_i^{(r)}$, we are reduced to verifying that the shifts by $\hbar$ that appear in the numerators of $\Phi_\mu^\lambda(E_i^{(r)})$ and (\ref{eq: monopole calculation 2}) agree.  This is equivalent to the equations (\ref{eq: shifts sigma}) for those $h\in Q_1$ with $i(h)=i$. The case $x=F_i^{(r)}$ is similar, and is equivalent to those equations where $o(h) = i$, proving the claim.

The elements $A_i^{(r)}, E_i^{(r)}, F_i^{(r)}$ generate $\mathbf{Y}_\mu[\equi_1,\ldots,\equi_N]$ as a Poisson algebra, under the Poisson bracket $\{a, b\} = \tfrac{1}{\hbar}(ab-ba)$.  Since $\cA_\hbar$ is almost commutative, it is closed under Poisson brackets.  It follows that there is a containment of graded $\CC[\hbar,\equi_1,\ldots,\equi_N]$--algebras
$$
\Phi_\mu^\lambda(\mathbf{Y}_\mu[\equi_1,\ldots,\equi_N]) \subseteq \sigma \bz^\ast(\iota_\ast)^{-1}(\cA_\hbar)
$$
Since $\sigma\bz^\ast(\iota_\ast)^{-1}: \cA_\hbar \hookrightarrow \widetilde{\cA}_\hbar$ is an embedding, the homomorphism $\overline{\Phi}_\mu^\lambda$ exists as claimed.
\end{proof}

The image of $\overline{\Phi}_\mu^\lambda$ is called the truncated shifted Yangian, and is denoted by $\mathbf{Y}_\mu^\lambda$.

We now give a generalization of \cite[Corollary B.28]{2016arXiv160403625B} and \cite[Theorem A]{Weekes} to BCFG types:

\begin{Theorem}\label{thm:quant}	
For any $\lambda \geq \mu$ we have an isomorphism $\mathbf{Y}_\mu^\lambda = \cA_\hbar$, and in particular $\mathbf{Y}_\mu^\lambda / \hbar \mathbf{Y}_\mu^\lambda \cong \overline{\cW}{}^{\underline{\lambda}}_\mu$.
\end{Theorem}

\begin{NB}
  It seems that the definition of $\mathbf{Y}_\mu^\lambda$ is missing.
  (May 10)
\end{NB}%

\begin{NB2}
May 15: Added the definition above.
\end{NB2}

\begin{proof}
$\mathbf{Y}_\mu^\lambda \rightarrow \cA_\hbar$ is injective by definition, so we must prove surjectivity.  When $\mu$ is dominant, this follows exactly as in the proof of \cite[Corollary B.28]{2016arXiv160403625B}.   To extend to case of general $\mu$, we follow the same strategy as the proof of \cite[Theorem 3.13]{Weekes}. First, we note that one can define shift homomorphisms for $\mathbf{Y}_\mu[t_1,\ldots,t_N]$ and $\cA_\hbar$, which are compatible as in \cite[Lemma 3.14]{Weekes}.  Second, we claim that $\cA_\hbar$ is generated by its subalgebras $\cA^\pm_\hbar$ corresponding to the loci $\cR^\pm$ lying over the positive and negative parts of the affine Grassmannian (c.f.~\cref{subsec:1mzastava}).  Assuming this claim for the moment, the proof of \cite[Theorem 3.13]{Weekes} now goes through.

To prove the claim about generators, consider the semigroups of integral points in chambers of the generalized root hyperplane arrangement for $\cA_\hbar$ (see \cite[Definition~5.2]{2016arXiv160103586B}).  The hyperplanes in our situation are of three types: (i) $w_{i,r} - w_{i,s} = 0$ for all $i\in I$ and $1\leq r, s \leq \bv_i$, (ii) $f_{ji} w_{i,r} - f_{ij} w_{j,s} = 0$ for any $c_{ij} \neq 0$  and $1\leq r \leq \bv_i, 1\leq s \leq \bv_j$, and (iii) $w_{i,r} = 0$ for any $W_i \neq 0$ and $1\leq r\leq \bv_i$.  Even if $W_i = 0$, we can always refine our arrangement by adding all hyperplanes $w_{i,r}$.  In this refined arrangement, any chamber is the product of its subcones of positive and negative elements. Thus we can choose generators for its semigroup of integral points which are each either positive or negative.  Since the spherical Schubert variety through a positive (resp.~negative) coweight lies inside $\Gr^+$ (resp.~$\Gr^-$), we can lift the above semigroup generators to algebra generators for $\cA_\hbar$ which each lie in one of $\cA^\pm_\hbar$.
This proves the claim.
\end{proof}

\begin{NB2}
May 15: Changed the above proof.  In the second part of the proof, regarding the generalized hyperplane arrangement, is the explanation clear enough? Add a reference to \cite[Proposition 6.8]{2016arXiv160103586B}?
\end{NB2}

\begin{NB}
  I have added the reference to \cite[Definition
  5.2]{2016arXiv160103586B} for the definition of generalized roots. I
  think that the explanation is enough. A serious reader must read
  \cite{Weekes} anyway.

  I also changed the generalized root of type (ii) slightly, as $d_i$,
  $d_j$ do not enter the definition of the Coulomb branch. The
  hyperplane itself remains the same, though. May 15.
\end{NB}%

%


\appendix
\section{A zastava space for $G_2$}\label{sec:app}

We give an explicit presentation of the coordinate ring of the zastava $Z^{\alpha_1+\alpha_2}$ of type $G_2$, thought of as a variety over a field of characteristic zero (for simplicity, we will simply work over $\CC$).  This presentation is similar to those for other rank 2 types given in \cite[Sections 5.5--5.8]{bdf}.  

Denote by $\g$ the Lie algebra of type $G_2$, and write $V(\lambda)$ for its irreducible representation of highest weight $\lambda$.  Following the notation \cite[Table 22.1]{MR1153249}, we pick a basis for the adjoint representation:
$$
V({\varpi_2}) \cong \g = \operatorname{span}_\CC \{ H_1, H_2, X_i, Y_i : 1\leq i \leq 6 \},
$$
Here $X_i, H_i, Y_i$ with $i=1, 2$ are the Chevalley generators with respect to the Cartan matrix $\begin{pmatrix} 2 & -3 \\ -1 & 2 \end{pmatrix}$.  Note this is the transpose of the convention taken in \cref{subsec:1mzastava}.  We define $X_3 = [X_1, X_2]$, $X_4 = \tfrac{1}{2} [X_1, X_3]$, $X_5 = -\tfrac{1}{3}[X_1, X_4]$, $X_6 = -[X_2, X_5]$ and similarly for the $Y_i$ (but with opposite signs).  In particular, $X_6$ is a highest weight vector.  Following \cite[pg. 354]{MR1153249}, we also pick a basis for the first fundamental representation:
$$
V({\varpi_1}) = \operatorname{span}_\CC \{ V_4, V_3, V_1, U, W_1, W_3, W_4\},
$$
where $V_4$ is a highest weight vector and $V_3 = Y_1 \cdot V_4$, $V_1 = - Y_2 \cdot V_3$, $U = Y_1 \cdot V_1$, $W_1 = \tfrac{1}{2} Y_1\cdot U$, $W_3 = Y_2 \cdot W_1$, and $W_4 = -Y_1\cdot W_3$.

Using the above notation, recall that $Z^{\alpha_1+\alpha_2}$ has a description as  Pl\"ucker sections \cite[Section 5]{mf}:  it is the space of pairs $v_{\varpi_i} \in V(\varpi_i)[z]$ for $i=1,2$ such that (a) the coefficient of $V_4$  in $v_{\varpi_1}$ (resp. $X_6$ in $v_{\varpi_2}$) is monic of degree one, (b) the coefficients of all other basis vectors have degree zero, and (c) certain Pl\"ucker-type relations must hold (see the proof below for certain cases).

\begin{Proposition}
Scheme-theoretically, $Z^{\alpha_1 + \alpha_2}$ is the set of pairs 
\begin{align*}
v_{\varpi_1} &= (z+A_1) V_4 + b_0 V_3 + b_2 V_1 + b_3 U + b_4 W_1, \\
v_{\varpi_2} &= (z+A_2) X_6 + b_1 X_5 + b_2 X_4 + b_3 X_3 + b_4 X_2
\end{align*}
whose coefficients satisfy
$$
b_0 b_1 = (A_2 - A_1 ) b_2, \ \ b_0 b_2 = (A_1 - A_2) b_3, \ \ b_0 b_3 = (A_1 - A_2) b_4, 
$$
$$
b_2^2 = - b_1 b_3, \ \  b_2 b_3 = - b_1 b_4, \ \ b_3^2 = b_2 b_4
$$

\end{Proposition}
\begin{NB2}
April 16: Notation changed slightly above
\end{NB2}
\begin{proof}
Fix a non-zero $\g$--invariant element $\Omega_2 \in \g \otimes \g$\footnote{For the purposes of our Sage calculation, we chose $\Omega_2$ corresponding to the trace form on $V(\varpi_1)$.}.  This can be considered as an operator on any $V(\lambda)\otimes V(\mu)$, and it distinguishes the canonical summand $V(\lambda+\mu) \subset V(\lambda)\otimes V(\mu)$ as an eigenspace \cite[Section 14.12]{Kac}.  

Consider an arbitrary pair $v_{\varpi_i} \in V(\varpi_i)[z]$ for $i=1,2$ satisfying the degree requirements a), b) above.  Using Sage, we compute the ideal defined by the above eigenvalue conditions for $\Omega_2$ applied to $v_{\varpi_i} \otimes v_{\varpi_j}$ where $1\leq i \leq j \leq 2$.  We find that this ideal has two primary components, which have dimensions 4 and 1, respectively. Since $Z^{\alpha_1 + \alpha_2}$ is a 4-dimensional irreducible closed subscheme living inside the vanishing locus of this ideal, it must correspond to the 4-dimensional primary component.  This yields the description claimed above.
\end{proof}

\begin{Remark}
Comparing with \cref{subsec:1mzastava} in the case $m=3$, we can identify the above coordinates with the generators of the Coulomb branch as follows:  $w_1 = -A_2, w_2 = -A_1, \sfy_{0,1} = b_0, \sfy_{1,0} = - b_1, \sfy_{1,1} = b_2, \sfy_{1,2} = b_3$ and $\sfy_{1,3} = b_4$.  
\end{Remark}
\begin{NB2}
April 16: added the above remark, and modified the one below
\end{NB2}

\begin{Remark}
To match the proposition with the conventions of \cite[Section 5.8]{bdf}, we take $\overline{w}_i = - A_1, \overline{w}_j = -A_2, \overline{y}_i = b_0$ and $\overline{y}_j = - b_1$ (we add overlines to avoid confusion with our notation for Coulomb branches).  The equation of the boundary of $Z^{\alpha_1+\alpha_2}$ is then
$$
- \frac{\overline{y}_i^3 \overline{y}_j}{(\overline{w}_i-\overline{w}_j)^3} = - \frac{b_0^3 b_1}{(A_1 - A_2)^3} =  b_4
$$
This is consistent with our comparison with the open zastava from \cref{subsec:1m}: by the previous remark $b_4 = \sfy_{1,3}$, which is invertible in $H^{G_\cO}_*(\cR)$. It is also easy to see that $H^{G_\cO}_*(\cR)$ is generated by the inverse element $\sfy_{-1,-3}$ together with $H^{G_\cO}_*(\cR^+)$, as expected.

\begin{NB2}Actually there seems to be a mistake exchanging $i$ and $j$ in \cite{bdf}.\end{NB2}

\begin{NB2} April 16: After discussing with Misha, there is no mistake in \cite{bdf}.  Rather they use $d_i$ for coroot lengths, not root \end{NB2}
\end{Remark}

\begin{NB2}
Should consider factorization structure following \cite{bdf}
\end{NB2}

\begin{NB}
    As far as I understand, we only need to identify what is the
    factorization morphism. It is just given by $(w_1, w_2)$, as it is
    so on the open locus. What else do we need ?
\end{NB}%

\begin{NB2}
April 16: I agree, we do not need more than what you said
\end{NB2}


\section{Fixed point sets}\label{sec:fixed}

Consider the category $\mathcal C$ of finitely generated right modules
of the quantized Coulomb branch $H^{G_\cO\rtimes\CC^\times}_*(\cR)$
such that (1) $\hbar\in H^*_{\CC^\times}(\mathrm{pt})$ acts by a
nonzero complex number, say $1$, and (2) it is locally finite over
$H^*_{G}(\mathrm{pt})$, hence it is a direct sum of generalized
simultaneous eigenspaces of $H^*_{G}(\mathrm{pt})$. (When we include
an additional flavor symmetry, we assume that the corresponding
equivariant parameter acts by a complex number.) One can apply
techniques of the localization theorem in equivariant ($K$)-homology
groups of affine Steinberg varieties in \cite{MR3013034} to study the
category $\mathcal C$. This theory, for the ordinary Coulomb branch,
will be explained elsewhere \cite{modules}. (See also
\cite{2016arXiv161106541W,2019arXiv190405415W} for another algebraic approach different from one in \cite{MR3013034}.)
It also works in our current setting. As a consequence, we have for example
\begin{Theorem}
  Let $\lambda\in\mathfrak t$.
  There is a natural bijection between
  \begin{itemize}
  \item simple modules in $\mathcal C$ such that one of eigenvalues
    above is given by evaluation
    $H^*_G(\mathrm{pt})\cong \CC[\mathfrak t]^\Weyl\to \CC$ at
    $\lambda$,
    \begin{NB}
      More general eigenvalues are translations of $\lambda$ by
      cocharacters of $T$,
    \end{NB}%
  \item simple perverse sheaves which appear, up to shift, in the
    direct image of constant sheaves on the fixed point subset
    $\cT^{(\lambda,1)}$ under the projection
    $\cT^{(\lambda,1)}\to \bN_\cK^{(\lambda,1)}$.
  \end{itemize}
\end{Theorem}

Here $\mathfrak t$ is the Lie algebra of a maximal torus of $G$,
$\Weyl$ is the Weyl group of $G$, and $(\lambda,1)$ is the element of
the Lie algebra of $T\times\CC^\times$, which acts on $\mathcal T$ and
$\bN_\cK$, as a subgroup of $G_\cO\rtimes\CC^\times$. Fixed point
subsets are written as $\cT^{(\lambda,1)}$, $\bN_\cK^{(\lambda,1)}$,
and the projection is the restriction of $\Pi\colon \cT\to\bN_\cK$.

\begin{NB}
The translation of $\lambda$ by a cocharacter $\mu$ of $T$ means as
follows. Recall $G = \prod \GL(V_i)$. So we have
$\lambda = \sum \lambda_i$, $\mu=\sum \mu_i$. We regard $\mu_i$ as an
element of the Lie algebra of $T(V_i)$. Recall that we have $d_i$,
appearing in the power of loop rotation action for the affine
Grassmannian $\Gr_{\GL(V_i)}$ (see \cref{section: loop
  rotation}). Then the translation of $\lambda$ by $\mu$ is the sum of
$\lambda_i + d_i \mu_i$.
\end{NB}%
\begin{NB}
  This shift comes from the finding a multiplicative subset satisfying
  the Ore condition from $H^*_G(\mathrm{pt})$: Since
  $H^*_G(\mathrm{pt})$ is not central, we do not have $f x = x f$ for
  $f\in H^*_G(\mathrm{pt})$ in general. We change $f x$ to $x' f'$,
  then $f'$ is a `shift' of $f$. It can be computed from (a corrected
  version of) \cite[Prop.~6.2]{2016arXiv160103586B}.
\end{NB}%

We study the fixed point set $\cT^{(\lambda,1)}$,
$\bN_\cK^{(\lambda,1)}$ in this section. For simplicity, we assume
$\lambda$ is a differential of a cocharacter, denoted by the same
symbol $\lambda$. (See \cref{rem:non-integral} for general case.)
Therefore we study the fixed point set with respect to a one parameter
subgroup $\tau\mapsto (\lambda(\tau),\tau)$.

\subsection{}

Consider the affine Grassmannian $\Gr_G$. We have an action of
$G_\cO\rtimes\CC^\times$ on $\Gr_G$ given by
$(h(z),\tau)\cdot [g(z)] = [h(z) g(z\tau)]$.
Take a cocharacter $\lambda\colon\CC^\times\to T$ and consider a
homomorphism
$\tau \mapsto (\lambda(\tau), \tau)\in T\times\CC^\times\subset
G_\cO\rtimes\CC^\times$, where $\lambda(\tau)$ is regarded as a constant
loop in $G_\cO$. Let
\begin{equation*}
  \Gr_G^{(\lambda(\tau),\tau)} \defeq
  \{ [g(z)]\in \Gr_G \mid (\lambda(\tau),\tau)\cdot [g(z)] = [g(z)] \}
\end{equation*}
be the fixed point set of $\lambda\times\id$ in $\Gr_G$. It consists
of equivalence classes $[g(z)]$ where
\begin{equation*}
  g(z) = \lambda(z)^{-1} \varphi(z)\quad
  \text{for a cocharacter $\varphi\colon\CC^\times\to G$}.
\end{equation*}
\begin{NB}
  We have $(\lambda(\tau),\tau)\cdot[\lambda(z)^{-1}\varphi(z)]
  = [\lambda(\tau)\lambda(z\tau)^{-1} \varphi(z\tau)]
  = [\lambda(z)^{-1} \varphi(z)\varphi(\tau)]
  = [\lambda(z)^{-1} \varphi(z)]$.
\end{NB}%
To see this let us identify $\Gr_G$ with $\Omega G_c$ the space of
polynomial based maps $(S^1,1)\to (G_c,1)$, where $G_c$ is a maximal
compact subgroup of $G$. Then $g\in\Omega G_c$ is fixed if and only if
$\lambda(\tau) g(z\tau) g(\tau)^{-1} \lambda(\tau)^{-1} = g(z)$. It means
that $z\mapsto \lambda(z) g(z)$ is a group homomorphism.

Alternatively the fixed point set can be identified as follows: Let
\begin{equation*}
  \begin{split}
  Z_{G_\cK}(\lambda(\tau),\tau) & = \text{centralizer of
    $(\lambda(\tau),\tau)$ in $G_\cK$} \\
  &= \{ g(z)\in G_\cK \mid \lambda(\tau) g(z\tau) \lambda(\tau)^{-1} = g(z)
  \}.
  \end{split}
\end{equation*}
Then $g(z=1)$ is well-defined and
$g(z) = z^{-\lambda} g(z=1) z^\lambda$, hence
$Z_{G_\cK}(\lambda(\tau),\tau)\cong G$ via $g(z)\mapsto g(1)$. (We
switch the notation from $\lambda(z)$ to $z^\lambda$.) Then the fixed
point set is
\begin{equation*}
  \bigsqcup_{\mu} Z_{G_\cK}(\lambda(\tau),\tau)\cdot [z^{-\lambda+\mu}],
\end{equation*}
where $\mu$ is a dominant coweight of $G$, and $z^{-\lambda+\mu}$ is
regarded as a point in $\Gr_G$. The
$Z_{G_\cK}(\lambda(\tau),\tau)$-orbit through $z^{-\lambda+\mu}$ is a
partial flag variety $G/P_\mu$, where $P_\mu$ is the parabolic
subgroup corresponding to $\mu$.

\subsection{}
More generally consider a homomorphism
$\tau\mapsto (\lambda(\tau), \tau^m)$ for $m\in\ZZ_{>0}$. We suppose
$G = \GL(V)$ and decompose $V = \bigoplus V(k)$ so that $\lambda(\tau)$
acts on $V(k)$ by $\tau^{k}\id_{V(k)}$. We consider $k$ modulo $m$ and
decompose $V$ as
\begin{equation*}
  V = V\{1\}\oplus\cdots \oplus V\{m\}, \qquad
  \text{where $V\{ k\} = \bigoplus_{l\equiv k\bmod m} V(l)$}.
\end{equation*}
Let $G' \defeq \GL(V\{1\})\times\cdots\times\GL(V\{m\})$.
Then $[g(z)]$ is fixed by $(\lambda(\tau),\tau^m)$ if and only if
$[g(z)] = ([g_1(z)],\dots,[g_m(z)]) \in \Gr_{G'}$ such that
\begin{gather*}
  g_k(z) = \lambda_k(z)^{-1} \varphi_k(z)\quad
  \text{for a cocharacter $\varphi_k\colon\CC^\times\to \GL(V\{k\})$}
  \quad (k=1,\dots, m).
\end{gather*}
Here $\lambda_k(z)$ is defined so that it acts by $z^{(l-k)/m}$ on $V(l)$.
\begin{NB}
  So $\lambda(z) \lambda_k(z^m)^{-1} = z^k\id$ on $V\{k\}$. Hence
  $(\lambda(\tau),\tau))\cdot [\lambda_k(z)^{-1} \varphi_k(z)]
  = [\lambda(\tau) \lambda_k(z\tau^m)^{-1} \varphi_k(z\tau^m)]
  = [ \tau^k \lambda_k(z)^{-1} \varphi_k(z)] = [\lambda_k(z)^{-1}\varphi_k(z)]$.
\end{NB}%
It is proved as follows. Take a based loop model $g\in\Omega G_c$. It
is fixed if and only if
$\lambda(\tau) g(z\tau^m) g(\tau^m)^{-1} \lambda(\tau)^{-1} =
g(z)$. Taking $\tau = \omega$, a primitive $m$-th root of unity, we
see that $g(z)$ preserves the decomposition
$V = V\{1\}\oplus\cdots\oplus V\{m\}$, hence it is in $\Gr_{G'}$. Let
$g_k(z)$ be the $k$-th component. Note that $\lambda(\tau)$ is
$\tau^k \lambda_k(\tau^m)$ on $V\{k\}$. Therefore we have
$\lambda_k(\tau^m) g_k(z\tau^m) g_k(\tau^m)^{-1}
\lambda_k(\tau^m)^{-1} = g_k(z)$. Hence $\lambda_k(z) g_k(z)$ is a
group homomorphism, which we denoted by $\varphi_k(z)$.

\begin{NB}
  This argument implicitly uses the fact that the centralizer of
  $\lambda(\omega)$ is connected. It is not true in general. Hence the
  following will be commented out.

It can be rephrased as follows: Note first that the action of
$(\lambda(\tau), \tau^m)$ can be also written as
$(\lambda(\tau),\tau^m)\cdot [g(z)] = [\lambda(\tau) g(z\tau^m)
\lambda(\tau)^{-1}]$.
We define
$G' \defeq \{ g\in G\mid \lambda(\omega)g\lambda(\omega)^{-1} = g\}$, where
$\omega$ is a primitive $m$-th root of unity. When $g\in G'$ the
adjoint action $\lambda(\tau)g\lambda(\tau)^{-1}$ is given by
$\lambda'(\tau^m) g \lambda'(\tau^m)^{-1}$ for some
$\lambda'\colon\CC^\times\to T$. Now $g(z) = \lambda'(z)^{-1}\varphi(z)$
with a cocharacter $\varphi\colon \CC^\times\to G'$ is fixed as
\begin{equation*}
  \begin{split}
  &\phantom{{}={}} [\lambda(\tau) g(z\tau^m) \lambda(\tau)^{-1}]
  = [\lambda'(\tau^m) g(z\tau^m) \lambda'(\tau^m)^{-1}]
  = [\lambda'(\tau^m) \lambda'(z\tau^m)^{-1}\varphi(z\tau^m) \lambda'(\tau^m)^{-1}]\\
  &= [\lambda'(z)^{-1}\varphi(z)].
  \end{split}
\end{equation*}
In the above situation we have
$G' = \GL(V\{1\})\times\cdots\times \GL(V\{m\})$. The above
$\lambda_k$, $\varphi_k$ ($1\le k\le m$) give the current $\lambda'$,
$\varphi$ as $\lambda' = \bigoplus_k \lambda_k$,
$\varphi = \bigoplus_k \varphi_k$.
Moreover this formulation makes sense for arbitrary $G$.

In the based loop model $g\in \Omega G_c$, it is fixed if and only if
$\lambda(\tau) g(z\tau^m) g(\tau^m)^{-1} \lambda(\tau)^{-1} = g(z)$. Taking
$\tau = \omega$, we see that $g(z) \in G'$. Then we can rewrite the
left hand side of the equation as
$\lambda'(\tau^m) g(z\tau^m) g(\tau^m)^{-1} \lambda'(\tau^m)^{-1}$. Hence we
see that $\lambda'(z) g(z)$ is a homomorphism as above.
\end{NB}%

Let 
$\lambda' \defeq \lambda_1\oplus\cdots\oplus\lambda_{m}$,
$\varphi \defeq \varphi_1\oplus\cdots\oplus \varphi_{m}$.
The connected component of $\Gr_G^{(\lambda(\tau),\tau^m)}$ containing
$[g(z)] = [\lambda'(z)^{-1}\varphi(z)]$ is a partial flag manifold
$G'/P_{\varphi}$ where $P_{\varphi}$ is a parabolic subgroup defined by
$\{ g\in G' \mid \exists\lim_{z\to 0}\varphi(z)^{-1}g\varphi(z)\}$.
\begin{NB}
  The action of $G'$ is given by
  $[\lambda'(z)^{-1} \varphi(z)] \mapsto [\lambda'(z)^{-1} g\varphi(z)]$
  ($g\in G'$). Then
  $[\lambda'(z)^{-1} g\varphi(z)] = [\lambda'(z)^{-1}\varphi(z)]$ if and
  only if $\varphi(z)^{-1} g\varphi(z) \in G'_\cO$.
\end{NB}%

Note that the decomposition $V = V\{1\}\oplus\cdots\oplus V\{m\}$
and the group $G'$ depends on the choice of $\lambda$. If we take
$\lambda = 1$ for example, we have $V = V\{m\}$ and $G' = G$.

Alternative description is as follows: Let
\begin{equation*}
  \begin{split}
    Z_{G_\cK}(\lambda(\tau),\tau^m) &=
    \text{the centralizer of $(\lambda(\tau),\tau^m)$ in $G_\cK$} \\
    &= \{ g(z)\in G_\cK \mid g(z) = z^{-\lambda'} g(z=1) z^{\lambda'},
    g(z=1)\in G'\}.
  \end{split}
\end{equation*}
It is isomorphic to $G'$
by $g(z) \mapsto g(z=1)\in G'$. Then the fixed point set is
\begin{equation*}
  \bigsqcup_\mu Z_{G_\cK}(\lambda(\tau),\tau^m) \cdot [z^{-\lambda'+\mu}],
\end{equation*}
where $\mu$ is a dominant cocharacter of $G'$, and the orbit
$Z_{G_\cK}(\lambda(\tau),\tau^m) \cdot [z^{-\lambda'+\mu}]$ is
isomorphic to the partial flag variety $G'/P_\mu$.

\begin{Remark}
  For general reductive groups $G$, the centralizer
  $Z_{G_\cK}(\lambda(\tau),\tau^m)$ could be
  disconnected. Nevertheless the description is still valid, if we
  replace $Z_{G_\cK}(\lambda(\tau),\tau^m)$ by its connected component
  $Z_{G_\cK}^0(\lambda(\tau),\tau^m)$.
  \begin{NB}
    I am not sure that $\lambda'$ does make sense in general. But the
    description is still valid, if we replace $z^{-\lambda'+\mu}$
    simply by $z^\mu$ and require that it is dominant with respect to
    $Z_{G_\cK}^0(\lambda(\tau),\tau^m)$.
  \end{NB}%
\end{Remark}

\subsection{}

Let us consider the case $I=\{1,2\}$, $c_{12} = -1$, $c_{21} = -m$
($m\in\ZZ_{>0}$) as in \cref{subsec:1m}. We have $z_1 = z = z_2^m$. We
consider the variety $\cT$, where we regard it as the space consisting of
\begin{itemize}
\item $[g_1(z_1)]\in \GL(V_1)((z_1))/\GL(V_1)[[z_1]]$,
\item $[g_2(z_2)]\in \GL(V_2)((z_2))/\GL(V_2)[[z_2]]$,
\item $B\in \Hom_{\CC((z_1))}(V_1((z_1)), V_2((z_2)))$ such that
  $g_2(z_2)^{-1} B g_1(z_1)$ is regular at $z_1=0$.
\end{itemize}
Here $V_2((z_2))$ is regarded as a $\CC((z_1))$-module via
$z_1 = z_2^m$. By the projection formula, we identify it with an
element in $\Hom_{\CC((z_2))}(V_1((z_2)), V_2((z_2)))
\cong \Hom_\CC(V_1,V_2)((z_2))$, and denote it by $B(z_2)$.
\begin{NB}
  Note
  $\CC((z_2)) = \CC((z_1))\oplus z_2\CC((z_1))\oplus\cdots \cdots
  z_2^{d-1}\CC((z_1))$. Therefore
  $B\in \Hom_{\CC((z_1))}(V_1((z_1)), V_2((z_2)))$ decomposes
  $B_0\oplus z_2 B_1\oplus\cdots\oplus z_2^{d-1} B_{d-1}$ with
  $B_i \in\Hom_{\CC((z_1))}(V_1((z_1)), V_2((z_1)))\cong \Hom(V_1,
  V_2)((z_1))$. If we write $B_i$ as $B_i(z_1)$, then $B$ as
  $\Hom_{\CC((z_2))}(V_1((z_2)), V_2((z_2))) \cong \Hom(V_1,
  V_2)((z_2))$ is written as $B(z_2) = B_0(z_1) + z_2 B_1(z_1)
  + \cdots + z_2^{d-1} B_{d-1}(z_1)$.
\end{NB}%
The action of
$(\GL(V_1)[[z_1]]\times \GL(V_2)[[z_2]])\rtimes \CC^\times$ on the
component $B(z_2)$ is given by
\begin{equation*}
  B(z_2) \mapsto h_2(z_2) B(z_2\tau) h_1(z_1)^{-1}\quad
  (h_1(z_1), h_2(z_2),\tau)\in
  (\GL(V_1)[[z_1]]\times \GL(V_2)[[z_2]])\rtimes \CC^\times.
\end{equation*}
Note that the loop rotation acts on $\GL(V_1)[[z_1]]$ by
$h_1(z_1)\mapsto h_1(\tau^m z_1)$ as $z_1 = z_2^m$. Also
$h_1(z_1)^{-1}$ is regarded as a function in $z_2$ via $z_1 = z_2^m$.

We take $\lambda_1\colon\CC^\times\to T(V_1)$,
$\lambda_2\colon \CC^\times\to T(V_2)$ as above, and consider the fixed
point set $\cT^{(\lambda_1(\tau),\lambda_2(\tau),\tau)}$ in $\cT$ with
respect to $\lambda_1\times\lambda_2\times\id$.
Then we have a decomposition
\begin{equation*}
  V_1 = V_1\{1\}\oplus\cdots\oplus V_1\{m\}
\end{equation*}
and $([g_1(z_1)], [g_2(z_2)])\in \Gr_{\GL(V_1)}\times \Gr_{\GL(V_2)}$
is given as
\begin{equation*}
  g_1(z_1) = \lambda'_1(z_1)^{-1} \varphi_1(z_1), \quad
  g_2(z_2) = \lambda_2(z_2)^{-1}\varphi_2(z_2)
\end{equation*}
for cocharacters
$\varphi_1\colon\CC^\times\to
\GL(V_1\{1\})\times\cdots\times\GL(V_1\{m\})$ and
$\varphi_2\colon\CC^\times\to \GL(V_2)$. Here $\lambda'_1$ is defined
from $\lambda_1$ as above.

\begin{Remark}\label{rem:non-integral}
  More generally we could study the fixed point set with respect to a
  cocharacter $\tau\mapsto (\lambda_1(\tau), \lambda_2(\tau), \tau^d)$ for
  $d \in\ZZ_{> 0}$. But the fixed point set will be just the union of
  $d$ copies of the fixed point set below, hence it does not yield a
  new space. On the other hand, this modification yields a new space
  when a quiver has a loop. See \cite{MR3013034}.
\end{Remark}

Let us consider the remaining component $B(z_2)$. It is fixed by the
action if and only if
\begin{equation*}
  B(z_2) = \lambda_2(\tau) B(z_2\tau) \lambda_1(\tau)^{-1}.
\end{equation*}
If we expand $B(z_2)$ as
$\cdots + B^{(-1)} z_2^{-1} + B^{(0)} + B^{(1)} z_2 + B^{(2)} z_2^2 +
\cdots$, this equation is equivalent to
\begin{equation*}
  B^{(n)} = \tau^n \lambda_2(\tau) B^{(n)} \lambda_1(\tau)^{-1}.
\end{equation*}
When we decompose $V_1$, $V_2$ as $\bigoplus V_1(k)$,
$\bigoplus V_2(k)$ as eigenspaces with respect to $\lambda_1(\tau)$,
$\lambda_2(\tau)$ as before, this equation means that $B^{(n)}$ sends
$V_1(i)$ to $V_2(i-n)$. In particular, $B^{(n)}$ must vanish if $|n|$
is sufficiently large, hence $B(z_2)$ is a Laurent polynomial. We see that
the evaluation $B(z_2=1)$ at $z_2 = 1$ \emph{does} make sense and is
equal to $\cdots + B^{(-1)} + B^{(0)} + B^{(1)} + \cdots$. Then
$B(z_2)$ is recovered from $B(z_2=1)$ by the formula
\begin{equation*}
  B(z_2) = \lambda_2(z_2)^{-1} B(z_2=1) \lambda_1(z_2).
\end{equation*}
Thus the fixed point set in
$\Hom_{\CC((z_2))}(V_1((z_2)), V_2((z_2)))$ is identified with the space
$B(z_2=1)\in \Hom_\CC(V_1, V_2)$.
\begin{NB}
Thus the fixed point set in
$\Hom_{\CC[[z_2]]}(V_1[[z_2]], V_2[[z_2]])$ is identified with the space
$B(z_2=1)\in \Hom_\CC(V_1, V_2)$ such that
\begin{equation*}
  B(z_2=1)( V_1(k))\subset \bigoplus_{l\le k} V_2(j).
\end{equation*}
\end{NB}%
\begin{NB}
  By its definition, we have decompositions
  $V_1\{0\} = \cdots \oplus V_1(0)\oplus V_1(m)\oplus \cdots$,
  $V_1\{1\} = \cdots \oplus V_1(1)\oplus V_1(m+1)\oplus \cdots$, and
  so on. Gradings on the summand $V_1\{k\}$ are indexed by $k + m\ZZ$.
\end{NB}%

Let us consider the condition that $g_2(z_2)^{-1} B(z_2) g_1(z_2^m)$
is regular at $z_2 = 0$ with
$g_1(z_1) = \lambda'_1(z_1)^{-1}\varphi_1(z_1)$,
$g_2(z_2) = \lambda_2(z_2)^{-1}\varphi_2(z_2)$. It is equivalent to
\begin{equation}\label{eq:1}
  \varphi_2(z_2)^{-1} B(z_2=1) \lambda_1(z_2) \lambda'_1(z_2^m)^{-1}\varphi_1(z_2^m)
\end{equation}
is regular at $z_2=0$. Note that $\lambda_1(z_2)\lambda'_1(z_2^m)^{-1}$ is
equal to $z_2^k$ on the summand $V_1\{k\}$.
We introduce a new grading on $V_1$, $V_2$ given by $\varphi_1$,
$\varphi_2$. For $V_2$, we define $V_2^\varphi(k)$ as the $\tau^k$
eigenspace with respect to $\varphi_2(\tau)$ as above. For $V_1$, let
us recall that $\varphi_1$ preserves the decomposition
$V_1 = V_1\{1\}\oplus\cdots\oplus V_1\{m\}$. Then we define
$V_1^\varphi(l)$ as the $\tau^{(l-k)/m}$ eigenspace with respect to
$\varphi_1(\tau)$ in $V_1\{k\}$, where $1\le k \le m$ is determined
so that $l\equiv k\mod m$. If $\varphi_1 = \lambda'_1$ (and hence
$g_1(z_1) = \id$), it is nothing but $V_1 = \bigoplus V_1(k)$.
Then \eqref{eq:1} is regular at $z_2$ if and only if
\begin{equation}\label{eq:2}
    B(z_2=1) (V_1^\varphi(k)) \subset \bigoplus_{l\le k} V_2^\varphi(l).
\end{equation}

\begin{NB}
  Thus it is more natural to consider the one parameter subgroup
  $\tau\mapsto \bigoplus \tau^l \id_{V_2^\varphi(l)}$. By its
  definition it is given by a shift of $\lambda$ by $m \mu$, and then
  conjugate by $G'$.
\end{NB}%

Thus connected components of the fixed point sets
$\cT^{(\lambda(\tau),\tau)}$ (as well as their projection to
$\bN_\cK^{(\lambda(\tau),\tau)}$) are almost the same as varieties
appeared in Lusztig's construction of canonical bases from quivers
\cite[\S1.5]{Lu-can2}, where the quiver has vertices $1_1$, \dots,
$1_m$, $2$ and arrows $1_k\to 2$. See \cref{fig:fixed} for $m=3$. Note
that this is different from the right quiver in \cref{fig:G2}. The
only differences from Lusztig's varieties are (1) the degree $k$
subspace, i.e., $V_1^\varphi(k)\oplus V_2^\varphi(k)$, might not be
concentrated at a single vertex, and (2) the degree $l$ subspace on the
vertex $1_k$ is only allowed when $k\equiv l\mod m$. But these
differences are \emph{superficial}. If the flag types at vertices are
the same and the conditions \eqref{eq:2} are the same, the grading is
not relevant. We get isomorphic varieties.

\begin{figure}[htbp]
    \centering
\begin{tikzpicture}[scale=1.2,
circled/.style={circle,draw
,thick,inner sep=0pt,minimum size=6mm},
squared/.style={rectangle,draw
,thick,inner sep=0pt,minimum size=6mm},
triplearrow/.style={
  draw=black!75,
  color=black!75,
  thick,
  double distance=3pt, 
  decoration={markings,mark=at position .75 with {\arrow[scale=.7]{>}}},
  postaction={decorate},
  >=stealth}, 
thirdline/.style={draw=black!75, color=black!75, thick, -
}
]
\node[circled] (vv1) at ( 0,1)  {$1_1$};
\node[circled] (vv2) at ( 0,0)  {$1_2$};
\node[circled] (vv3) at ( 0,-1) {$1_3$};
\node[circled] (vv4) at ( 1,0)  {$2$}
edge [<-,thick] (vv1)
edge [<-,thick] (vv2)
edge [<-,thick] (vv3);
\end{tikzpicture}
\caption{The quiver appearing in the fixed point set}
    \label{fig:fixed}
\end{figure}
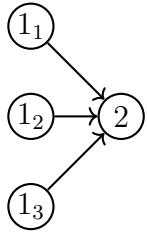

\subsection{}

The analysis of the fixed point set in the previous subsection can be
applied to general cases. The final claim that components of the fixed
point set $\cT^{(\lambda(\tau),\tau)}$ are isomorphic to Lusztig's
varieties remains true if the quiver (corresponding to
\cref{fig:fixed}) has no loop, in particular, for type $BCFG$.
Therefore
\begin{Theorem}\label{thm:ADE}
  Consider the quantized Coulomb branches $\cAh$ of type $BCFG$
  with $W=0$. Then we have a natural bijection between
  \begin{itemize}
  \item simple objects in the category $\mathcal C$ such that their
    eigenvalues are evaluations at cocharacters of $T$,
  \item canonical base elements of weight
    $-\sum \dim (V_i\{k\}) \alpha_{i_k}$ in the lower triangular part
    $\mathbf U_q^-$ of the quantized enveloping algebra of type $ADE$.
  \end{itemize}
\end{Theorem}

Here $i$ runs over the set of vertices of the original quiver, and
$k$ runs from $1$ to $d_i$.
Concretely the correspondence between types is $B_n\mapsto A_{2n-1}$,
$C_n\mapsto D_{n+1}$, $F_4\mapsto E_6$, $G_2\mapsto D_4$.
\begin{NB}
as a vertex $i$
in type $BCFG$ splits into $d_i$ vertices in type $ADE$.
\end{NB}%

\begin{Remark}
  Note that the canonical base elements in the above theorem are in
  bijection also to simple objects in the category $\mathcal C$ (with
  the same constraint) of the quantized Coulomb branch of type $ADE$
  by the same analysis of the fixed point set as above.
  Recall that the quantized Coulomb branch $\cAh$ is a quotient of the
  shifted Yangian of type $BCFG$ or $ADE$, the same type as
  quiver. Therefore we have a bijective correspondence between simple
  modules in quotients of shifted Yangian of type $BCFG$ and of $ADE$.
  This result reminds us the result of Kashiwara, Kim and Oh
  \cite{MR3898986}, where a similar bijection was found between simple
  finite dimensional modules of quantum affine algebras of types $B_n$
  and $A_{2n-1}$.
\end{Remark}


\section{A second definition}
\label{section: second definition}
In this section we present a second possible definition for a Coulomb branch associated to a quiver gauge theory with symmetrizers.  In the case when the Cartan matrix satisfies assumption (\ref{assumption on Cartan}), this second definition agrees with that given in \cref{sec:definition}.  But in general this is not the case.  We note that this second definition applies to theories which are not of quiver type.

\subsection{Covers of disks} For each $k \in \ZZ_{>0}$ consider the formal disc $D_k = \Spec \CC[[x^k]] $.  If $k | \ell$, there is a map
\begin{equation}
\rho_{k | \ell} : D_k  \longrightarrow D_\ell
\end{equation}
corresponding to the inclusions of rings $\CC[[x^\ell]] \hookrightarrow \CC[[x^k]]$. Similarly there are maps between the corresponding formal punctured discs, which we also denote $\rho_{k|\ell} : D_k^\ast \rightarrow D_\ell^\ast$ by abuse of notation.  These maps are equivariant for the $\CC^\times$--action by loop rotation, $\tau: x^k\mapsto \tau^k x^k$.
\begin{NB2}
June 25: I hope there will be no confusion with the discs from earlier in the paper.  At least for $D_i$ the index $i\in I$, not $i\in \ZZ$.
\end{NB2}

\subsection{General definition}

Fix a pair $(\altG, \altN)$, consisting of $\altG = \prod_{k=1}^d G_k$ a product of complex connected reductive groups, and $\altN = \bigoplus_{k=1}^d \bN_k$ a direct sum of complex finite-dimensional representations of $\altG$.  In addition, we assume that $G_k$ acts \emph{trivially} on $\bN_j$, unless $j | k$.

Given such a pair $(\altG,\altN)$, we define $\cR_{\altG,\altN}$ to be the moduli space of triples $(\cP_\bullet, \varphi_\bullet, s_\bullet)$, where $\cP_\bullet = (\cP_1,\ldots,\cP_d)$, $\varphi_\bullet = (\varphi_1,\ldots,\varphi_d)$, and $s_\bullet = (s_1,\ldots, s_d)$ satisfy
\begin{enumerate}
\item[(a)] $\cP_k$ is a principal $G_k$--bundle over $D_k$,

\item[(b)] $\varphi_k$ is a trivialization of $\cP_k$ over $D_k^\ast$,

\item[(c)] $s_k$ is a section of the associated bundle
$$
s_k \in \Gamma\Big( D_k, \Big( \prod_{k | \ell} \rho_{k| \ell}^\ast \cP_\ell\Big) \times^{\prod_{k|\ell} G_\ell} \bN_k \Big),
$$
such that it is sent to a regular section of the trivial bundle under the trivialization $\prod_{k| \ell} \rho_{k| \ell}^\ast \varphi_\ell$ over $D_k^\ast$.
\end{enumerate}

As usual we also define a larger moduli space $\cT_{\altG, \altN}$ by dropping the extension conditions in (c).

The group $\altGO = \prod_{k=1}^d G_k [[x^k]]$ acts on $\cR_{\altG,\altN}$ by changing $\varphi_\bullet$.  There is also an action of $\CC^\times$, acting by loop rotation of the discs $D_k$ as in the previous section.  We can define a convolution product on $H^{\altGO}_\ast(\cR_{\altG, \altN})$ just as in \cite{2016arXiv160103586B}.  By the argument in \ref{subsection: convolution product}, it is a commutative ring, and we define the Coulomb branch
$$
\cM_C(\altG,\altN) \stackrel{def}{=} \Spec H^{\altGO}_\ast(\cR_{\altG, \altN})
$$
It has a deformation quantization defined by $H^{\altGO\rtimes \CC^\times}(\cR_{\altG, \altN})$, and in particular a Poisson structure. 

The arguments from cite \cite{2016arXiv160103586B} apply with small modifications to $\cM_C(\altG,\altN)$.  In particular it is finite type, integral, normal, and generically symplectic.  One useful observation in modifying the proofs is the following:

\begin{Remark}
Suppose that $\altG = G_\ell$ consists of a single factor, and define its representation $\bN' = \bigoplus_{k | \ell} \bN_k^{\oplus(\ell / k)}$.  Then $\cM_C(\altG, \altN)$ is isomorphic to the usual Coulomb branch $\cM_C(G_\ell, \bN')$ as defined in \cite{2016arXiv160103586B}.  This comes from the fact that there is an isomorphism
$\bN_k [[x^k]] = \bigoplus_{0\leq a < \ell/k} x^{a k} \bN_k[[x^\ell]]  \cong \bN_k[[x^\ell]]^{\oplus(\ell / k)}$ as representations of $G_\ell[[x^\ell]]$.
\end{Remark}

\begin{NB2}
It is straightforward to write down the monopole formula for general $(\altG, \altN)$.  Is this worth including?
\end{NB2}

\begin{NB2}
July 4: Added a brief mention of twisted monopole formula in the quiver case, below.  
\end{NB2}

\subsection{The quiver case}
\label{subsection: second definition}

As in \cref{subsection: A valued quiver}, consider a valued quiver  associated to a symmetrizable Cartan matrix $(c_{ij})_{i,j\in I}$.   Also choose symmetrizers $(d_i)\in \ZZ^d_{>0}$.  Recall that we denote \(
   g_{ij} = \gcd(|c_{ij}|, |c_{ji}|),
\)
\(
   f_{ij} = |c_{ij}|/g_{ij}
\)
when $c_{ij} < 0$.  It is not hard to see that $d_i$ must be a multiple of $f_{ji}$ for any $c_{ij} <0$, so we may define integers $d_{ij}$ by the rule $d_i = d_{ji} f_{ji}$.  They satisfy $d_{ij} = d_{ji}$.

\begin{Remark}
In fact, $\operatorname{lcm}(d_i, d_j) = d_i f_{ij} = d_j f_{ji}$ and $\operatorname{gcd}(d_i, d_j) = d_{ij} = d_{ji}$.
\end{Remark}

Choose vector spaces $V_i$ and $W_i$ for each $i \in I$.  Given these choices, we define a pair $(\altG, \altN)$ according to the following rules:
\begin{align}
G_k & = \prod_{\substack{i \in I, \\ d_i = k}} \operatorname{GL}(V_i), \\
\bN_k & = \bigoplus_{\substack{i\in I, \\ d_i = k}} \Hom(W_i, V_i) \oplus \bigoplus_{\substack{j \rightarrow i, \\ d_{ij} = k }} \CC^{g_{ij}} \otimes_\CC \Hom(V_j, V_i)
\end{align}
Then $\altN$ is a representation of $\altG$ in the natural way, and satisfies our assumption from the beginning of the previous section. 
By tracing through the definition one can see that the moduli space $\cR_{\altG, \altN}$ parametrizes:
\begin{itemize}
\item a rank $\bv_i$ vector bundle $\cE_i$ over $D_{d_i}$ together with a
  trivialization $\varphi_i\colon \cE_i|_{D_{d_i}^*} \to
  V_i\otimes_\CC \shfO_{D_{d_i}^*}$ for $i\in I$,
\item a homomorphism $s_i\colon W_i\otimes_\CC\shfO_{D_{d_i}}\to \cE_i$
  such that $\varphi_i\circ (s_i|_{D_{d_i}^*})$ extends to $D_{d_i}$ for
  $i\in I$,
\item a homomorphism
  $s_{ij}\in \CC^{g_{ij}}\otimes_\CC  \Hom_{\shfO_{D_{d_{ij}}}}(\rho^\ast_{d_{ij}|d_j}\cE_j,\rho^\ast_{d_{ij}| d_i}\cE_i)$
  such that
  $(\rho^\ast_{d_{ij} | d_i}\varphi_i) \circ (s_{ij}|_{D_{d_{ij}}^*}) \circ
  (\rho^\ast_{d_{ij}|f_j}\varphi_j)^{-1}$ extends to $D_{d_{ij}}$, where $c_{ij} < 0$ and
  there is an arrow $j\to i$ in the quiver.
\end{itemize}


\subsection{Comparison}
\label{subsection: comparison}
We now compare with the construction from \cref{subsection: A moduli space}.  For this it suffices to understand the case of a single edge $j\rightarrow i$.  As explained in Section \ref{section: loop rotation}, we can $\CC^\times$--equivariantly identify $D_i \cong D_{d_i}$ via $z_i \mapsto x^{d_i}$, $D_j \cong D_{d_j}$ via $z_j \mapsto x^{d_j}$, and $D\cong D_{d_i f_{ij}} = D_{d_j f_{ji}}$ via $z\mapsto x^{d_i f_{ij}}$.   We also denote $D' = D_{d_{ij}} = D_{d_{ji}}$.  Then there are commutative diagrams of discs and their corresponding rings, as in \cite[\S 4.2]{arXiv181209663}\footnote{We thank an anonymous referee for pointing out this reference.}:
\begin{equation*}
\begin{tikzcd}[row sep = scriptsize]
& D & & & & \CC[[x^{d_i f_{ij}}]] \ar[dr, hookrightarrow] \ar[dl, hookrightarrow] &\\ 
D_i \ar[ur, "\pi_{ji}"] & & D_j \ar[ul, "\pi_{ij}"'] & & \CC[[x^{d_i}]] \ar[dr, hookrightarrow] & & \CC[[x^{d_j}]] \ar[dl, hookrightarrow] \\
& D_i \times_D D_j \ar[ul] \ar[ur] & & & & \CC[[x^{d_i}, x^{d_j}]] \ar[d,hookrightarrow] & \\
& D' \ar[u] & & & & \CC[[x^{d_{ij}}]]&
\end{tikzcd}
\end{equation*}
Both squares are Cartesian, while the inclusion $\CC[[x^{d_i}, x^{d_j}]] \hookrightarrow \CC[[x^{d_{ij}}]]$ is of finite codimension over $\CC$. We also note that $\CC[[x^{d_i f_{ij}}]] = \CC[[x^{d_i}]] \cap \CC[[x^{d_j}]]$.

For brevity, let us denote the covering maps $\rho_{ij} = \rho_{d_{ij} | d_j} : D' \rightarrow D_j$ and $\rho_{ji} = \rho_{d_{ji} | d_i}: D' \rightarrow D_i$.  
Then the difference between the two constructions from \cref{subsection: A moduli space} and \cref{subsection: second definition} is simply in the definition the section $s_{ij}$: whether it lies in
\begin{equation}
\label{eq: comparison}
\CC^{g_{ij}} \otimes_\CC \Hom_{\cO_D}( \pi_{ij \ast} \cE_j, \pi_{ji \ast} \cE_i) \quad \text{ or } \quad \CC^{g_{ij}} \otimes_\CC \Hom_{\cO_{D'}}( \rho_{ij}^\ast \cE_j, \rho_{ji}^\ast \cE_i)
\end{equation}
We now reformulate both sides in terms of the above power series rings, ignoring the tensor product with $\CC^{g_{ij}}$ in each case.  Denote by $E_i$ the $\CC[[x^{d_i}]]$--module corresponding to $\cE_i$, and by $E_j$ the $\CC[[x^{d_j}]]$--module corresponding to $\cE_j$. Then on the one hand, the left side of (\ref{eq: comparison}) corresponds to
\begin{equation*}
\Hom_{\CC[[x^{d_i f_{ij}}]]} (E_j, E_i) \ \cong \ \Hom_{\CC[[x^{d_i}]]}\big( \CC[[x^{d_i}, x^{d_j}]] \otimes_{\CC[[x^{d_j}]]} E_j, E_i \big)
\end{equation*}
On the other hand, the right side of (\ref{eq: comparison}) corresponds to
\begin{equation*}
\Hom_{\CC[[x^{d_{ij}}]]} \big( \CC[[x^{d_{ij}}]]\otimes_{\CC[[x^{d_j}]]} E_j, \CC[[x^{d_{ij}}]]\otimes_{\CC[[x^{d_i}]]} E_i \big)
\end{equation*}
\begin{equation*}
\cong \Hom_{\CC[[x^{d_i}]]} \big( \CC[[x^{d_{ij}}]]\otimes_{\CC[[x^{d_j}]]} E_j, E_i \big)
\end{equation*}
For this isomorphism we use the fact that induction and coinduction of modules between the rings $A =\CC[[x^{d_i}]] \hookrightarrow B =\CC[[x^{d_{ij}}]]$ are isomorphic as functors: there is an isomorphism of left $B$--modules $\Hom_A(B,A) \cong B$ (equivariant up to a grading shift, for the loop $\CC^\times$--action).

Thus we see that the difference between the two sides of (\ref{eq: comparison}), and thus between our two constructions, is captured by the finite codimension inclusion of rings
\begin{equation*}
\CC[[x^{d_i}, x^{d_j}]] \hookrightarrow \CC[[x^{d_{ij}}]]
\end{equation*}
Note that this map is an isomorphism if and only if $f_{ij} =1$ or $f_{ji} =1$.

\begin{Theorem}
For a general valued quiver, if the assumption (\ref{assumption on Cartan}) holds then our constructions from \cref{subsection: A moduli space} and \cref{subsection: second definition} are isomorphic.  In particular, this is the case in all finite types.
\end{Theorem}

\subsection{Twisted monopole formula}
The previous section shows that the twisted monopole formula applies to $\cR_{\altG, \altN}$ in the case when assumption (\ref{assumption on Cartan}) holds.  But in fact it is not hard to see that the twisted monopole formula is valid for $\cR_{\altG, \altN}$ even when this assumption does not hold.  More precisely, Proposition \ref{prop: eq: poincare polynomial} is valid for $\cR_{\altG, \altN}$ in all types, with the same expression for $d_\lambda$ from \cref{subsection: twisted monopole formula}.

\begin{NB2}
July 12: Updated this section to include rank 2 calculations, which were previously in Section 3(ii)
\end{NB2}

The twisted monopole formula is related to the following generalization of the calculations from \cref{sec:example}: for an arbitrary rank 2  Cartan matrix we find that
\begin{equation*}
  \bz^*(w_1) = w_1, \quad \bz^*(w_2) = w_2, \quad
  \bz^*(\sfy_{a,b}) = (w_1-w_2)^{g_{12} \cdot \max(f_{12}b - f_{21} a,0)}\sfu_{a,b}.
\end{equation*}
Indeed, the fiber of $\cT_{\altG, \altN}$ over $(a,b) \in \Gr_G = \ZZ^2$ is
\begin{equation*}
\CC^{g_{12}} \otimes_\CC x^{a d_1 - b d_2} \Hom_{\CC}( V_2, V_1)[[x^{d_{12}}]],
\end{equation*}
while the fiber of $\cR_{\altG, \altN}$ is its intersection with $\CC^{g_{12}} \otimes_\CC \Hom_\CC(V_2, V_1) [[x^{d_{12}}]]$.  The contribution to $\bz^*(\sfy_{a,b})$ above is the Euler class of the quotient, recalling that $d_1 = d_{12} f_{21}$ and $d_2 = d_{12} f_{12}$.

Take $(a_0,b_0)\in\ZZ^2$ such that
$f_{12} b_0 - f_{21} a_0 = 1$. Then we have
$\sfy_{f_{12},f_{21}}\sfy_{-f_{12},-f_{21}} = 1$, 
$(w_1 - w_2)^{g_{12}} = \sfy_{a_0,b_0}\sfy_{-a_0,-b_0}$. Hence we have
\(
H^{G_{\bullet, \cO}}_*(\cR_{\altG,\altN}) \cong \CC[w_1,\sfy_{f_{12},f_{21}}^{\pm},
\sfy_{a_0,b_0},\sfy_{-a_0,-b_0}].
\)
Therefore the Coulomb branch is $\BA\times\BA^\times \times  \BA^2 /( \ZZ/g_{12} \ZZ)$.


\subsection*{Acknowledgments}

We are grateful to our anonymous referees for their very helpful suggestions.
H.N.~thanks B.~Leclerc for explanations of \cite{MR3660306} and
subsequent developments over years.
He also thanks
M.~Finkelberg,
R.~Fujita, and
D.~Muthiah
for useful discussion.
A part of this paper was written while H.N.~was visiting the Simons
Center for Geometry and Physics. He wishes to thank its warm
hospitality.
A.W.~thanks A.~Braverman, M.~Finkelberg and B.~Webster for helpful discussions. 

The research of H.N.~was supported in part by the World Premier
International Research Center Initiative (WPI Initiative), MEXT,
Japan, and by JSPS Grant Numbers 16H06335, 19K21828.
This research of A.W.~was supported in part by Perimeter Institute for Theoretical Physics. Research at Perimeter Institute is supported by the Government of Canada through Innovation, Science and Economic Development Canada and by the Province of Ontario through the Ministry of Research, Innovation and Science.


\bibliographystyle{myamsalpha}
\bibliography{nakajima,mybib,coulomb,symmtype}
\end{document}